\documentclass{birkjour}

\usepackage[noadjust]{cite}
\usepackage{xcolor}
\RequirePackage[all]{xy}

\usepackage{amssymb}
\usepackage{amsmath}
\usepackage{dsfont}

\newtheorem{theorem}{Theorem}[section]

\newtheorem{definition}[theorem]{Definition}
\newtheorem{example}[theorem]{Example}

\newtheorem{lemma}[theorem]{Lemma}

\newtheorem{proposition}[theorem]{Proposition}
\newtheorem{remark}[theorem]{Remark}

\numberwithin{equation}{section}
\newcommand{\duer}{\mathbin{\raisebox{3pt}{\varhexstar}\kern-3.70pt{\rule{0.15pt}{4pt}}}\,}

\usepackage{enumitem}  
\usepackage{calc}
\setlist{labelindent=1pt,itemsep=.5em}
\setlist[itemize]{leftmargin=1.2cm}
\setlist[enumerate]{itemindent=0em,leftmargin=1.2cm}
\setlist[enumerate,1]{label={\upshape(\roman*)}}

\usepackage{ifpdf}
\ifpdf
  \usepackage[colorlinks=true,linkcolor=blue,citecolor=red, final,backref=page,hyperindex]{hyperref}
\else
  \usepackage[colorlinks,final,backref=page,hyperindex,hypertex]{hyperref}
\fi

\begin{document}

%
%
%
%
%
%
%
%
%

\title[On Completeness of  $n$-ary Hom-Nambu and Hom-Lie Superalgebras]
 {On Completeness of $n$-ary Hom-Nambu and Hom-Lie Superalgebras}

\author[M. R. Farhangdoost]{Mohammad Reza Farhangdoost}
%
\address{%
Department of Mathematics, College of Sciences, Shiraz University, P.O. Box 71457-44776, Shiraz, Iran}
\email{farhang@shirazu.ac.ir}
%


\author[M. R. Hafezi]{Mohammad Reza Hafezi}
\address{%
Department of Mathematics, College of Sciences, Shiraz University, P.O. Box 71457-44776, Shiraz, Iran}
\email{m.hafez@hafez.shirazu.ac.ir}

\author[S. Silvestrov]{Sergei Silvestrov}
\address{%
Department of Business and Mathematics,
M\"{a}lardalen University, Box 883, 72123 V\"{a}ster{\aa}s, Sweden}
\email{sergei.silvestrov@mdu.se}
\thanks{Corresponding author: Sergei Silvestrov (sergei.silvestrov@mdu.se)}

\subjclass{Primary 17B61, 17D30; Secondary 17B65, 17B68, 17B70}
\keywords{\(n\)-ary Hom-Nambu superalgebras;
\(n\)-Hom-Lie superalgebras;
complete Hom-Nambu superalgebras;
complete Hom-Lie superalgebras;
derivations;
low-dimensional classification}
\date{\today}
\begin{abstract}
We introduce and study completeness for multiplicative $n$-ary Hom-Nambu superalgebras. Because an $n$-ary Hom-Nambu bracket is not necessarily totally super-skew-symmetric, we define its center as the intersection of its positional centers. A multiplicative $n$-ary Hom-Nambu superalgebra is complete when its center is trivial, and every $\alpha^{k+1}$-derivation is inner for all $k \geq 0$. We show that this notion reduces to the usual completeness for multiplicative $n$-Hom-Lie superalgebras. Furthermore, we establish a completeness criterion for surjective brackets with a zero twisting map and provide a complete $n$-ary Hom-Nambu superalgebra that is not an $n$-Hom-Lie superalgebra. We also study the direct sums of complete $n$-Hom-Lie superalgebras and the behavior of centers under twisting. Finally, starting from a multiplicative Hom-Lie superalgebra, we consider recursively induced multiplicative $n$-ary Hom-Nambu superalgebras. When the twisting map is surjective, we prove that triviality of the center is preserved and reflected in this construction.  We also observe that every binary $\alpha^k$-derivation satisfies the corresponding relative $\alpha^k$-derivation identity for the induced bracket. Finally, we determine the complete members in the selected low-dimensional Hom-Lie and $3$-Hom-Lie superalgebra classifications.
\end{abstract}

\maketitle




\section{Introduction}\label{sec1}
Hom-type algebras emerged as natural generalizations and deformations of classical algebraic structures,
where the defining identities are twisted using linear maps. In Hom-algebra structures, the compositions of multilinear operations with twisting linear maps allow non-trivial ways to deform and relate various algebraic structures and lead to the construction of new algebraic systems relevant to geometry, physics, and deformation theory. This concept was introduced by Hartwig, Larsson, and Silvestrov \cite{hartwiglarssonSilvestrov20032006} in 2003 (first appeared as a Lund University preprint in 2003, then as ArXiv.org preprint in August 2004, and finally as a journal publication in 2006), through their study of $\sigma$-derivations and quasi-deformations of Lie algebras,
leading to the notions of quasi-Lie, quasi-Hom-Lie, and Hom-Lie algebras.
Subsequent works developed a comprehensive framework for such Hom-algebra structures as Hom-associative, Hom-Lie, and Hom-Leibniz algebras, Hom-coalgebras, Hom-bialgebras and Hom-Hopf algebras and their deformation theory
\cite{LarssonSilvJA2005:QuasiHomLieCentExt2cocyid,LarssonSilv20042005CA2007:quasidefsl2,MakhlSilv2006GLTA2008-Homalgstr,MakhSilv2007FM2010-NotesformdefHomassHomLie,MakhSil200709HomHopfGLTBSpringer,MakhSilvJAA2010HomAlgHomCoalg,RichardSilvJA2008:quasiLiesigderCtpm1,RichardSilvestrovGLTMPBSpr2009:QuasiLieHomLiesigmaderiv,yau2008,yau2009}.

In Hom-Lie algebras, the bilinear product satisfies the non-twisted skew-symmetry property, as in Lie algebras, whereas the Hom-Lie algebra Jacobi identity has three terms twisted by a single linear map, reducing to the Lie algebra Jacobi identity when the twisting linear map is the identity map.
Color Hom-Lie algebras ($\Gamma$-graded $\epsilon$-Hom-Lie algebras) and more general color quasi-Lie algebras and color quasi-Leibniz algebras (including in particular the Hom-Lie superalgebras for grading by the abelian group $\Gamma=\mathbb{Z}_2$), were introduced by Silvestrov with his PhD students Daniel Larsson and Gunnar Sigurdsson in 2004--2006 in \cite{LarssonSilv2005:QuasiLieAlg,LarssonSilv:GradedquasiLiealg,SigSilv:CzechJP2006:GradedquasiLiealgWitt}, as the Hom-algebra structures naturally encompassing within the same structural framework the quasi-hom-Lie algebra structures of quasi-deformations and discretizations of Lie algebras of vector fields by $\sigma$-derivations from \cite{hartwiglarssonSilvestrov20032006}, and the color Lie algebras or $\Gamma$-graded $\epsilon$-Lie algebras (a direct generalization of Lie superalgebras to general grading studied both in mathematics and in theoretical physics since the 1970s).  In quasi-Lie Hom-algebras, the skew-symmetry and Jacobi identity are twisted by deforming twisting linear maps, and the Jacobi identity in quasi-Lie and quasi-Hom-Lie algebras generally contains six twisted triple bracket terms.

Hom-Lie superalgebras, that is the color Hom-Lie algebras with grading abelian group $\mathbb{Z}_2$ and standard ($\mathbb{Z}_2$-graded parity sign) commutation factor, have been further studied in \cite{ammarmakhlouf2010super}.
They have been investigated from both algebraic and physical perspectives \cite{guan2017,guan2019,fan2021}.
The classification of low-dimensional multiplicative Hom-Lie superalgebras was given in \cite{wang2016},
serving as a foundation for studying derivations, completeness, and structural decomposition
\cite{armakan2020complete,farhangdoost2023}.
Further generalizations include Hom-color and BiHom-Lie superalgebras,
which unify grading and twisting phenomena within a common categorical setting
\cite{ArmakanSilvFarhangTJM2019EnvAlgColHomLie,ArmakanSilvFarhangGMJ2021ExtHomLieColAlg,BakayokoSilvestrov-SPROMS2020-MultnHomLiecolalg, BakayokoSilv-AfM2021-HomleftSymcolHomtridendColAlg}.

The $n$-ary generalizations of Lie algebras, known as Filippov or Nambu-Lie algebras, were initiated by Nambu \cite{nambu1973} in the context of generalized Hamiltonian mechanics, and later formalized by Filippov \cite{filippov1985}.
These algebras are characterized by an $n$-linear skew-symmetric bracket that satisfies the Filippov identity, which is a higher-arity analog of the Jacobi identity.
They have found significant applications in mathematical physics,
especially in the study of Nambu-Poisson manifolds and higher order gauge symmetries
\cite{takhtajan1994,takhtajan1995,daletskii1997,deazcarraga2010,awataLiMincYoneyaJHEP2001-quantnambubr,bai2010,bai2011}.
The structural and representation theories of $n$-Lie algebras, as well as their cohomology and central extensions, have been systematically developed in many subsequent works.

Higher-arity Hom-algebra structures, such as $n$-Hom-Lie algebras and other $n$-ary Hom-algebras of Lie and associative types, were first introduced by Ataguema, Makhlouf and Silvestrov in~\cite{AtMaSi:GenNambuAlg}.
In higher-arity Hom-algebra structures, the defining identities are twisted by one or several linear maps,
yielding a broader class of deformations, such as $n$-ary Hom-algebra structures generalizing the $n$-ary algebras of Lie type including $n$-ary Nambu algebras, $n$-ary Nambu-Lie algebras and $n$-ary Lie algebras, and~$n$-ary algebras of associative type, including $n$-ary totally associative and $n$-ary partially associative algebras.
The structural properties, cohomologies, representations, central extensions and various constructions and examples for $n$-Hom-Nambu-Lie algebras, Hom-superalgebras and color Hom-algebras were considered in
\cite{AbdaouiMabrMakhl2014ArXiv-CohomlHomLeibnAryHomNambuLieSuperAlg,AmmarMabMakh2011-naryhomrep,ArnlindMkhSilvJMP2010:ternaryHomNambuLieINdHomLie,arnlindMakhlSilvJMP2011-ConstrnLienHomNambuLie,BakayokoSilvestrov-SPROMS2020-MultnHomLiecolalg,BakayokoSilv-AfM2021-HomleftSymcolHomtridendColAlg,KitMakSil-GMJ2016nhominduced,YauGenCom,YauHomNambuLie}.

In~\cite{kms:narygenBiHomLieBiHomassalgebras2020}, $n$-ary generalizations of BiHom-Lie and BiHom-associative algebras were considered. Generalized derivations of $n$-BiHom-Lie algebras were studied in~\cite{BenAbdeljElhamdKaygorMakhl201920GenDernBiHomLiealg}, and generalized derivations of multiplicative $n$-ary Hom-$\Omega$ color algebras were studied in~\cite{BeitesKaygorodovPopovBMMSS2019-GenDernaryHomOmegacolalg}.
The cohomology of Hom-Leibniz and $n$-ary Hom-Nambu-Lie superalgebras was considered in
~\cite{AbdaouiMabrMakhl2014ArXiv-CohomlHomLeibnAryHomNambuLieSuperAlg}.
The generalized derivations and Rota--Baxter operators of $n$-ary Hom-Nambu superalgebras were considered in~\cite{MabroukNcibSilvAACA2021-GenDerRotaBaxterOpsnaryHomNambuSuperal}.
The construction of $3$-Hom-Lie algebras based on $\sigma$-derivation and involution was studied in~\cite{AbramovSilvestrov:3homLiealgsigmaderivINvol}.
Multiplicative $n$-Hom-Lie color algebras were considered in
~\cite{BakayokoSilvestrov-SPROMS2020-MultnHomLiecolalg}.

An essential concept in the theories of Lie and Hom-Lie superalgebras is \textit{completeness}.
A Hom-Lie superalgebra is said to be complete if its center is trivial and all $\alpha^{k+1}$-derivations are inner for each $k\geq 0$.
This property, first studied in the classical Lie setting, was extended to Hom-algebra structures in \cite{armakan2020complete}, where the necessary and sufficient conditions for completeness and its decomposition into direct sums of complete ideals were established.
The completeness of BiHom-Lie superalgebras was studied in \cite{fan2021}. Related aspects of derivations and inner derivations for Hom-Lie color algebras were considered in \cite{ArmakanSilvFarhangGMJ2021ExtHomLieColAlg}, whereas further structural properties of Hom-Lie superalgebras were investigated in \cite{guan2019}. In the non-super-skew-symmetric Hom-Nambu setting, the position occupied by the central elements is essential. This motivates the introduction
of positional centers and their intersection as the appropriate notion of
center. For \(n\)-Hom-Lie superalgebras, total super-skew-symmetry makes
all positional centers coincide, so this definition reduces to the usual
one.

Derivations and generalized derivations are another central theme in this context.
They naturally appear in connection with cohomology, deformations, and representations of Hom-type algebras.
The theory of generalized derivations in Lie and $n$-Lie algebras was developed in
\cite{leger2000,zhang2010,chen2013,kaygorodov2012,kaygorodov2016}.
For Hom-type structures, derivations and Rota-Baxter operators were investigated in
\cite{bakayoko2014a,bakayoko2014b,BakayokoSilvestrov-SPROMS2020-MultnHomLiecolalg,BeitesKaygorodovPopovBMMSS2019-GenDernaryHomOmegacolalg,
BenAbdeljElhamdKaygorMakhl201920GenDernBiHomLiealg,
MabroukNcibSilvAACA2021-GenDerRotaBaxterOpsnaryHomNambuSuperal}, where the relationships between generalized derivations, Hom-Nambu systems, and Rota-Baxter operators are clarified.
These developments are further connected to Hom-dendriform and Hom-pre-Lie algebras
\cite{ma2017a,ma2017b} and to color Hom-Poisson algebras \cite{bakayoko2014a,bakayoko2014b}.

A particularly transparent completeness criterion arises when the
twisting map is zero and the bracket is surjective. This criterion also
provides complete \(n\)-ary Hom-Nambu superalgebras that are not
\(n\)-Hom-Lie superalgebras. We also examine completeness under direct
decompositions into Hom-ideals.
The interplay between derivations, centers, and twisting maps naturally motivates the study of completeness for \(n\)-ary Hom-Nambu and  $n$-ary Hom-Lie superalgebras.
Although completeness has been thoroughly explored in the binary case, its behavior in higher-arity settings remains largely unexplored.
The goal of this article is to investigate the structural stability of completeness
under the transition from binary to $n$-ary Hom-Lie superalgebras, building on the results and constructions in \cite{arnlindMakhlSilvJMP2011-ConstrnLienHomNambuLie,armakan2020complete,farhangdoost2023,MabroukNcibSilvAACA2021-GenDerRotaBaxterOpsnaryHomNambuSuperal}.

The recursive higher-arity construction considered in this paper
generally produces \(n\)-ary Hom-Nambu superalgebras rather than
\(n\)-Hom-Lie superalgebras. For these induced brackets, we study the
behavior of positional centers and establish a relative derivation identity,
without claiming a complete correspondence between the binary and
higher-arity derivation spaces. The general results are complemented by
explicit computations in selected low-dimensional classifications.

The remainder of this paper is organized as follows. Section~\ref{sec2-Prelim} reviews the fundamental concepts of Hom-Lie and \(n\)-Hom-Lie superalgebras, along with examples. Section~\ref{sec3-ComplnaryHomNambuSuperalg} develops the structural framework of Hom-subalgebras, Hom-ideals, centers, and centralizers, and introduces the notion of completeness in the \(n\)-ary setting. The behavior of the completeness property under suitable morphisms and twisting operations is considered, several results on twisting, direct sums, and recursively induced Hom-Nambu structures are presented, and the implications for derivation algebras and the decomposition of complete \(n\)-Hom-Lie superalgebras are discussed. In Section~\ref{sec4-ComplLowdimnHomLieSuperalg}, the complete low-dimensional Hom-Lie superalgebras are determined by computing their centers, derivations, and inner derivations, and the $2$-dimensional $3$-Hom-Lie case is classified and analyzed for completeness. Finally, Section~\ref{sec5-Conclusions} presents the conclusions and some future directions.

\section{Preliminaries}\label{sec2-Prelim}
In this section, we recall some basic notions and results that will be used throughout this paper.
For completeness, we briefly review the fundamental definitions of Hom-Lie superalgebras, their derivations, centroids and ideals.
We also recall related constructions that will be extended to the \( n \)-ary case.
Unless otherwise stated, all linear spaces are considered over a field \( \mathbb{K} \) of characteristic zero.
\begin{definition} [\cite{MabroukNcibSilvAACA2021-GenDerRotaBaxterOpsnaryHomNambuSuperal}]	A \emph{Hom-Lie superalgebra} is a $\mathbb{Z}_2$-graded linear space
$	\mathfrak{g} = \mathfrak{g}_{\bar{0}} \oplus \mathfrak{g}_{\bar{1}}$
	over a field $\mathbb{K}$, equipped with a bilinear map
		\(
		[\cdot,\cdot] : \mathfrak{g} \times \mathfrak{g} \longrightarrow \mathfrak{g}\),
		such that \([\mathfrak{g}_{\bar i},\mathfrak{g}_{\bar j}] \subseteq \mathfrak{g}_{\bar i+\bar j}\) for all \(\bar i, \bar j \in \mathbb{Z}_2
		\) {\rm (}"even" bilinear map{\rm )},
	 and a linear map
		\(
		\alpha : \mathfrak{g} \longrightarrow \mathfrak{g}\)
		such that \(\alpha(\mathfrak{g}_{\bar i}) \subseteq \mathfrak{g}_{\bar i}\) for all \(\bar i \in \mathbb{Z}_2
		\) {\rm (}"even" linear map{\rm )},
	satisfying the following identities for all $x,y,z \in H(\mathfrak{g}) = \mathfrak{g}_{\bar 0}\cup \mathfrak{g}_{\bar 1}$ {\rm (}homogeneous $x,y,z\in \mathfrak{g}${\rm )}{\rm :}
	\begin{align}
		& [x,y] = -(-1)^{|x||y|}[y,x], & \textnormal{(super skew-symmetry identity)} \label{superskewsymid}\\
		&\sum_{\circlearrowleft_{x,y,z}} (-1)^{|x||z|}\,[\alpha(x),[y,z]] = 0 , \label{superHomJacobiid}
		& \textnormal{(super Hom-Jacobi identity)}
	\end{align}
	where $|x| \in \mathbb{Z}_2$ denotes the parity {\rm (}grading degree{\rm )} of $x$, and
	$\sum\limits_{\circlearrowleft_{x,y,z}}$ denotes the summation over the cyclically permuted triple $x,y,z$.
\end{definition}

\begin{remark}
An ordinary Lie superalgebra is a Hom-Lie superalgebra with $\alpha = \mathrm{id}$.
Hom-Lie algebras are Hom-Lie superalgebras with $\mathfrak{g}_{\bar 1}={0}$.
The ordinary Lie algebras are Hom-Lie algebras with $\alpha = \mathrm{id}$, or equivalently, Hom-Lie superalgebras with $\alpha = \mathrm{id}$ and $\mathfrak{g}_{\bar 1}={0}$.
\end{remark}

\begin{definition}
A Hom-Lie superalgebra $(\mathfrak{g}, [\cdot,\cdot], \alpha)$ is called
multiplicative if
\[
\alpha([x,y]) = [\alpha(x), \alpha(y)], \quad \forall\, x,y \in \mathfrak{g}.
\]
\end{definition}

\begin{remark}
Multiplicativity, while useful in some special constructions such as cohomology and in the construction of a special type of multiplicative Hom-Lie superalgebras from ordinary Lie superalgebras using the composition of the Lie superalgebra product with the multiplicative twisting map, is at the same time a rather restrictive condition in general, singling out a special subclass of Hom-Lie superalgebras.
\end{remark}

\begin{remark} In any Hom-Lie superalgebra, in terms of the linear maps
$\operatorname{ad}_x \in \operatorname{Lin}(\mathfrak{g})$ defined for each $x \in \mathfrak{g}$ by
\[
\operatorname{ad}_x(y) = [x,y], \quad \forall\, y \in \mathfrak{g},
\]
using super skew-symmetry \eqref{superskewsymid}, the super-Hom-Jacobi identity \eqref{superHomJacobiid} can be rewritten as
\[
\operatorname{ad}_{[x,y]}(\alpha(z)) \;=\;
\operatorname{ad}_{\alpha(x)} \circ \operatorname{ad}_y (z)
\;-\; (-1)^{|x||y|}\, \operatorname{ad}_{\alpha(y)} \circ \operatorname{ad}_x (z).
\]
for all \( x , y , z \in H(\mathfrak{g}). \)
\end{remark}

Before proceeding, we illustrate the definitions using a simple example.

\begin{example} \label{osposp} Let \(\mathfrak{osp}(1,2)\) denote the complex orthosymplectic Lie
superalgebra
\(
\mathfrak{osp}(1,2)=V_{\bar 0}\oplus V_{\bar 1},
\)
where the even part \(V_{\bar 0}\) has basis \(\{H,X,Y\}\), with
\[
H=
\begin{pmatrix}
1&0&0\\
0&0&0\\
0&0&-1
\end{pmatrix},
\qquad
X=
\begin{pmatrix}
0&0&1\\
0&0&0\\
0&0&0
\end{pmatrix},
\qquad
Y=
\begin{pmatrix}
0&0&0\\
0&0&0\\
1&0&0
\end{pmatrix},
\]
and the odd part \(V_{\bar 1}\) has basis \(\{F,G\}\), with
\[
F=
\begin{pmatrix}
0&0&0\\
1&0&0\\
0&1&0
\end{pmatrix},
\qquad
G=
\begin{pmatrix}
0&1&0\\
0&0&-1\\
0&0&0
\end{pmatrix}.
\]
The Lie superbracket is the standard matrix supercommutator
\[
[A,B]=AB-(-1)^{|A||B|}BA
\]
for homogeneous elements \(A,B\in\mathfrak{osp}(1,2)\), extended by
bilinearity.
For \(\lambda\in\mathbb{R}^{*}=\mathbb{R}\setminus\{0\}\), define the
even linear map
\(
\alpha_\lambda:
\mathfrak{osp}(1,2)\longrightarrow\mathfrak{osp}(1,2)
\)
by
\[
\begin{aligned}
\alpha_\lambda(H)&=H,
&
\alpha_\lambda(X)&=\lambda^2X,
&
\alpha_\lambda(Y)&=\frac{1}{\lambda^2}Y,\\
\alpha_\lambda(F)&=\frac{1}{\lambda}F,
&
\alpha_\lambda(G)&=\lambda G.
\end{aligned}
\]
Define the Hom-Lie superbracket by
\[
[x,y]_{\alpha_\lambda}
=
\alpha_\lambda([x,y]),
\qquad
\forall\,x,y\in\mathfrak{osp}(1,2).
\]
Its nonzero defining relations are
\[
\begin{aligned} \relax
[H,X]_{\alpha_\lambda}
&=2\lambda^2X,
&
[H,Y]_{\alpha_\lambda}
&=-\frac{2}{\lambda^2}Y,
&
[X,Y]_{\alpha_\lambda}
&=H,\\
[Y,G]_{\alpha_\lambda}
&=\frac{1}{\lambda}F,
&
[X,F]_{\alpha_\lambda}
&=\lambda G,
&
[H,F]_{\alpha_\lambda}
&=-\frac{1}{\lambda}F,\\
[H,G]_{\alpha_\lambda}
&=\lambda G,
&
[G,F]_{\alpha_\lambda}
&=H,
&
[G,G]_{\alpha_\lambda}
&=-2\lambda^2X,\\
[F,F]_{\alpha_\lambda}
&=\frac{2}{\lambda^2}Y.
\end{aligned}
\]
The remaining products are determined by super-skew-symmetry and
bilinearity. Hence,
\(
\mathfrak{osp}(1,2)_\lambda
=
\bigl(
\mathfrak{osp}(1,2),
[\cdot,\cdot]_{\alpha_\lambda},
\alpha_\lambda
\bigr)
\)
is a multiplicative Hom-Lie superalgebra.

For \(\lambda=1\), we have
\(
\alpha_1=\operatorname{id}_{\mathfrak{osp}(1,2)},
\)
and therefore
\(\mathfrak{osp}(1,2)_1\) coincides with the original Lie superalgebra \(\mathfrak{osp}(1,2)\).
\end{example}

The above example shows how a classical Lie superalgebra can be deformed into a family of Hom-Lie superalgebras via an endomorphism.
We now provide the classification of all $2$-dimensional multiplicative Hom-Lie superalgebras
up to isomorphism.

\begin{proposition}
Let $\lambda,\mu\in\mathbb{R}^{*}$. Then,
\(
\mathfrak{osp}(1,2)_\lambda
\cong
\mathfrak{osp}(1,2)_\mu
\)
as Hom-Lie superalgebras if and only if
\(
\mu=\lambda
\ \text{or}\
\mu=\lambda^{-1}.
\)
\end{proposition}

\begin{proof}
Assume that
\(
\varphi:
\mathfrak{osp}(1,2)_\lambda
\longrightarrow
\mathfrak{osp}(1,2)_\mu
\)
is an isomorphism of Hom-Lie superalgebras. Then,
\(
\varphi\circ\alpha_\lambda
=
\alpha_\mu\circ\varphi.
\)
Since $\varphi$ is even, it restricts to an isomorphism of the odd
parts. Relative to the basis $\{F,G\}$, the restrictions of the
twisting maps are
\[
\left.\alpha_\lambda\right|_{V_{\bar 1}}
=
\begin{pmatrix}
\lambda^{-1}&0\\
0&\lambda
\end{pmatrix},
\qquad
\left.\alpha_\mu\right|_{V_{\bar 1}}
=
\begin{pmatrix}
\mu^{-1}&0\\
0&\mu
\end{pmatrix}.
\]
Hence, these two matrices are similar and therefore have the same
characteristic polynomials. Consequently,
\[
(t-\lambda^{-1})(t-\lambda)
=
(t-\mu^{-1})(t-\mu),
\]
which implies
\(
\lambda+\lambda^{-1}
=
\mu+\mu^{-1}.
\)
Equivalently,
\(
(\lambda-\mu)(\lambda\mu-1)=0.
\)
Thus,
\(
\mu=\lambda
\) or
\(\mu=\lambda^{-1}.
\)

Conversely, if $\mu=\lambda$, then the identity map is an isomorphism.
Suppose that $\mu=\lambda^{-1}$ and define the even linear map
$\varphi:\mathfrak{osp}(1,2)\longrightarrow\mathfrak{osp}(1,2)$ by
\[
\begin{aligned}
\varphi(H)&=-H,
&
\varphi(X)&=-Y,
&
\varphi(Y)&=-X,\\
\varphi(F)&=G,
&
\varphi(G)&=-F.
\end{aligned}
\]
A direct verification using the defining relations of
$\mathfrak{osp}(1,2)$ shows that $\varphi$ is a Lie superalgebra
automorphism. Furthermore,
\[
\varphi\circ\alpha_\lambda
=
\alpha_{\lambda^{-1}}\circ\varphi.
\]
Since
\(
[x,y]_{\alpha_\lambda}
=
\alpha_\lambda([x,y]),
\)
for all $x,y\in\mathfrak{osp}(1,2)$, it follows that
\[
\begin{aligned}
\varphi([x,y]_{\alpha_\lambda})
&=
\varphi\bigl(\alpha_\lambda([x,y])\bigr) =
\alpha_{\lambda^{-1}}\bigl(\varphi([x,y])\bigr)\\
&=
\alpha_{\lambda^{-1}}
\bigl([\varphi(x),\varphi(y)]\bigr) =
[\varphi(x),\varphi(y)]_{\alpha_{\lambda^{-1}}}.
\end{aligned}
\]
Therefore, $\varphi$ is an isomorphism from
$\mathfrak{osp}(1,2)_\lambda$ onto
$\mathfrak{osp}(1,2)_{\lambda^{-1}}$.
\end{proof}

\begin{theorem}[\cite{wang2016}] \label{2dimen}
Every $2$-dimensional complex multiplicative Hom-Lie superalgebra is isomorphic to one of the following
non-isomorphic Hom-Lie superalgebras. Each algebra is denoted by
$\mathfrak{g}^k_{i;j}$, where $i = \dim \mathfrak{g}_{\bar 0}$,
$j = \dim \mathfrak{g}_{\bar 1}$, and $k$ is an indexing parameter.
\begin{enumerate}[label=\textup{\arabic*.}]
	\item $\mathfrak{g}^1_{0;2}:$ an abelian Hom-Lie superalgebra.
	
	\item $\mathfrak{g}^2_{1;1}:$ an abelian Hom-Lie superalgebra.
	
	\item $\mathfrak{g}^{3,b}_{1;1}:$
	\(
	[e_0,e_1] = e_1,  [e_1,e_1] = 0,
	\alpha = \begin{pmatrix}
		1 & 0 \\[4pt]
		0 & b
	\end{pmatrix}.
	\)
	$\mathfrak{g}^{3,b}_{1;1}$ is isomorphic to $\mathfrak{g}^{3,b'}_{1;1}$
	if and only if $b = b'$.
	
	\item $\mathfrak{g}^{4,a}_{1;1}:$
	\(
	[e_0,e_1] = e_1,  [e_1,e_1] = 0, \quad
	\alpha = \begin{pmatrix}
		a & 0 \\[4pt]
		0 & 0
	\end{pmatrix},  a \neq 0,1.
	\)
	$\mathfrak{g}^{4,a}_{1;1}$ is isomorphic to $\mathfrak{g}^{4,a'}_{1;1}$
	if and only if $a = a'$.
	
	\item $\mathfrak{g}^{5,b}_{1;1}:$
	\(
	[e_0,e_1] = 0,  [e_1,e_1] = e_0,
	\alpha = \begin{pmatrix}
		b^2 & 0 \\[4pt]
		0 & b
	\end{pmatrix},  b \neq 0.
	\)
	$\mathfrak{g}^{5,b}_{1;1}$ is isomorphic to $\mathfrak{g}^{5,b'}_{1;1}$
	if and only if $b = b'$.
\end{enumerate}
\end{theorem}
In what follows, we extend the binary Hom-Lie superalgebra structure to the $n$-ary setting.
We recall some fundamental notions of $n$-ary Hom-Nambu and Hom-Lie superalgebras that generalize classical Nambu and Filippov algebras to the superalgebraic Hom-algebra framework.
Throughout this section, $N = N_{\bar 0} \oplus N_{\bar 1}$ denotes a $\mathbb{Z}_2$-graded linear space
over a field $\mathbb{K}$ of characteristic zero.

\begin{definition}[\cite{MabroukNcibSilvAACA2021-GenDerRotaBaxterOpsnaryHomNambuSuperal}]
An $n$-ary Hom-Nambu superalgebra $(N, [\cdot, \ldots, \cdot], \tilde{\alpha})$
is a triple consisting of a linear superspace $N = N_{\bar 0} \oplus N_{\bar 1}$,
an even $n$-linear map
\(
[\cdot, \ldots, \cdot] : N^n \to N
\)
such that
\(
[N_{\bar{k}_1},\ldots,N_{\bar{k}_n}]
\subseteq
N_{\bar{k}_1+\cdots+\bar{k}_n},
\qquad
\forall\,\bar{k}_1,\ldots,\bar{k}_n\in\mathbb{Z}_2,
\)
and a family
\(
\widetilde{\alpha}
=
(\alpha_i)_{1\leq i\leq n-1}
\)
of even linear maps
\(
\alpha_i:N\to N
\)
satisfying
\begin{equation}
	\begin{aligned}
		& [\alpha_1(x_1), \ldots, \alpha_{n-1}(x_{n-1}), [y_1, \ldots, y_n]] \\ \label{shni}
		& = \sum_{i=1}^n (-1)^{|X| \, |Y|_{i-1}}
		[\alpha_1(y_1), \ldots, \alpha_{i-1}(y_{i-1}),
		[x_1, \ldots, x_{n-1}, y_i], \\
		& \hspace{6cm} \alpha_i(y_{i+1}), \ldots, \alpha_{n-1}(y_n)],
	\end{aligned}
\end{equation}
for all $(x_1, \ldots, x_{n-1}) \in H(N)^{n-1}$ and $(y_1, \ldots, y_n) \in H(N)^n$,
where
\[
|X| = \sum_{k=1}^{n-1} |x_k|,
\qquad
|Y|_{i-1} = \sum_{k=1}^{i-1} |y_k|.
\]

Identity \eqref{shni} is called the \emph{super-Hom-Nambu identity}.
\end{definition}

\begin{definition}
An $n$-ary Hom-Nambu superalgebra $(N, [\cdot, \ldots, \cdot], \tilde{\alpha})$
is called an \emph{$n$-Hom-Lie superalgebra} if its bracket
$[\cdot, \ldots, \cdot]$ is super-skewsymmetric, that is,
for all $1 \leq i \leq n-1$ one has
\begin{equation}
	[x_1, \ldots, x_i, x_{i+1}, \ldots, x_n]
	= -(-1)^{|x_i||x_{i+1}|}
	[x_1, \ldots, x_{i+1}, x_i, \ldots, x_n].
\end{equation}

\end{definition}

\begin{remark}
When the maps $(\alpha_i)_{1 \leq i \leq n-1}$ are all identity maps, one recovers the classical $n$-ary Nambu superalgebras.
\end{remark}
\begin{definition}
A multiplicative $n$-ary Hom-Nambu superalgebra (resp. multiplicative
$n$-Hom-Lie superalgebra) is an $n$-ary Hom-Nambu superalgebra (resp.
$n$-Hom-Lie superalgebra) $(N, [\cdot, \ldots, \cdot], \tilde{\alpha})$ with
$\tilde{\alpha} = (\alpha_i)_{1 \leq i \leq n-1}$ such that
$\alpha_1 = \cdots = \alpha_{n-1} = \alpha$ and satisfying
\begin{equation}
	\alpha([x_1, \ldots, x_n]) = [\alpha(x_1), \ldots, \alpha(x_n)],
	\quad \forall\, x_1, \ldots, x_n \in N.
\end{equation}

For brevity, we denote the multiplicative $n$-ary Hom-Nambu superalgebra as
$(N, [\cdot, \ldots, \cdot], \alpha)$, where $\alpha : N \to N$ is an even linear map.
By a slight abuse of notation, an element $X \in N^n$ will be written as
$X = (x_1, \ldots, x_n)$ and $\alpha(X)$ denotes
$(\alpha(x_1), \ldots, \alpha(x_n))$.

\end{definition}

\begin{definition}
Let $(N,[\cdot,\ldots,\cdot],\widetilde{\alpha})$ and
$(N',[\cdot,\ldots,\cdot]',\widetilde{\alpha}')$ be two
$n$-ary Hom-Nambu superalgebras, where
\(
\widetilde{\alpha}=(\alpha_i)_{1\leq i\leq n-1},
\widetilde{\alpha}'=(\alpha_i')_{1\leq i\leq n-1}.
\)
An even linear map
\(
f:N\longrightarrow N'
\)
is called a morphism of \(n\)-ary Hom-Nambu superalgebras if
\(
f([x_1,\ldots,x_n])
=
[f(x_1),\ldots,f(x_n)]'
\)
for all homogeneous elements \(x_1,\ldots,x_n\in N\), and
\(
f\circ\alpha_i=\alpha_i'\circ f
\)
for $1\leq i\leq n-1$.
A morphism \(f:N\to N'\) is called an isomorphism if it is a bijection. In this case, the two \(n\)-ary Hom-Nambu superalgebras are said to be isomorphic.
A morphism from an \(n\)-ary Hom-Nambu
superalgebra to itself is called an endomorphism. A bijective
endomorphism is called an automorphism.
\end{definition}

Let \((N,[\cdot,\ldots,\cdot],\widetilde{\alpha})\) and \((N',[\cdot,\ldots,\cdot]',\widetilde{\alpha}')\) be two \(n\)-ary Hom-Nambu superalgebras, where \(\widetilde{\alpha}=(\alpha_i)_{1\leq i\leq n-1}\) and \(\widetilde{\alpha}'=(\alpha_i')_{1\leq i\leq n-1}\). An even linear map \(f:N\longrightarrow N'\) is called a (morphism of \(n\)-ary Hom-Nambu superalgebras) if \(f([x_1,\ldots,x_n]) = [f(x_1),\ldots,f(x_n)]'\) for all homogeneous elements \(x_1,\ldots,x_n\in N\), and \(f\circ\alpha_i=\alpha_i'\circ f\) for \(1\leq i\leq n-1\). A morphism \(f:N\to N'\) is called an isomorphism if it is a bijection. In this case, the two \(n\)-ary Hom-Nambu superalgebras are said to be isomorphic.
Morphisms from \(n\)-ary Hom-Nambu superalgebras to themselves are called endomorphisms, and the bijective endomorphisms are called automorphisms.  The set of all endomorphisms of an \(n\)-ary Hom-Nambu superalgebra \(N\) is denoted by \(\operatorname{End}(N)\), and the group of all automorphisms of \(N\) is denoted by \(\operatorname{Aut}(N)\). To avoid ambiguities in the notation, the set of all linear maps on a linear superspace \(N\) is denoted by \(\operatorname{Lin}(N)\). Note that \(\operatorname{End}(N)\subseteq \operatorname{Lin}(N)\) for a given \(n\)-ary Hom-Nambu superalgebra \(N\). This inclusion is strict for most \(n\)-ary Hom-Nambu superalgebras, with \(\operatorname{End}(N)\) typically being smaller than \(\operatorname{Lin}(N)\) because of the restrictive nature of the conditions of multiplicativity of endomorphisms with respect to \(n\)-ary multiplication and compatibility with the twisting maps \(\widetilde{\alpha }\).

\begin{definition}
[\cite{MabroukNcibSilvAACA2021-GenDerRotaBaxterOpsnaryHomNambuSuperal}]
Let $(N,[\cdot,\dots,\cdot],\alpha)$ be a multiplicative
$n$-ary Hom-Nambu superalgebra, and let $k\geq0$.
A homogeneous linear map $D\in \operatorname{Lin}(N)$ is said to have
parity \(|D|= \bar j\in\mathbb Z_2\) if
\[
D(N_{\bar i})\subseteq N_{\bar i+\bar j},
\qquad \forall\, \bar i\in\mathbb Z_2.
\]
Such a homogeneous linear map $D$ is called an
\emph{$\alpha^k$-derivation} if
\(
D\circ\alpha=\alpha\circ D,
\) (that is, \(D\) commutes with \(\alpha\)),
and for all homogeneous $x_1,\dots,x_n\in N$,
\begin{align*}
D\bigl([x_1,\dots,x_n]\bigr)
={}&
\sum_{i=1}^{n}
(-1)^{|D|\cdot(|x_1|+\cdots+|x_{i-1}|)}
\\
&\quad
\bigl[
\alpha^k(x_1),\dots,\alpha^k(x_{i-1}),
D(x_i),
\alpha^k(x_{i+1}),\dots,\alpha^k(x_n)
\bigr].
\end{align*}
The linear space of all
\(\alpha^k\)-derivations of parity \( \bar j\in\mathbb Z_2\) is denoted
\(
\operatorname{Der}_{\alpha^k}(N)_{\bar j}.
\)
We then set
\(
\operatorname{Der}_{\alpha^k}(N)
=
\operatorname{Der}_{\alpha^k}(N)_{\bar 0}
\oplus
\operatorname{Der}_{\alpha^k}(N)_{\bar 1}.
\)
Throughout this paper, we use the convention
\(
\alpha^0=\operatorname{id}_N.
\)
\end{definition}

\begin{definition}
For \( X = (x_1, \ldots, x_{n-1}) \in H(N)^{n-1} \) satisfying \( \alpha(X) = X \)
and for any integer \( k \geq 0 \), we define the linear map
\(\operatorname{ad}^{k}_{X} \in \operatorname{Lin}(N)\) by
\[
\operatorname{ad}^{k}_{X}(y) = [x_1, \ldots, x_{n-1}, \alpha^{k}(y)],
\qquad \forall\, y \in N.
\]
\end{definition}

\begin{lemma}[\cite{MabroukNcibSilvAACA2021-GenDerRotaBaxterOpsnaryHomNambuSuperal}]
Let
\(
D\in \operatorname{Der}_{\alpha^k}(N)
\)
and
\( D'\in \operatorname{Der}_{\alpha^{k'}}(N)
\)
be homogeneous, where $k,k'\geq 0$. Then their supercommutator
\[
[D,D']
=
D\circ D'
-
(-1)^{|D||D'|}D'\circ D
\]
belongs to
\(
\operatorname{Der}_{\alpha^{k+k'}}(N).
\)
Consequently, if
\[
\operatorname{Der}(N)
=
\bigoplus_{k\geq 0}\operatorname{Der}_{\alpha^k}(N),
\]
define, for \(\bar j\in\mathbb{Z}_2\),
\[
\operatorname{Der}(N)_{\bar j}
=
\bigoplus_{k\geq0}
\operatorname{Der}_{\alpha^k}(N)_{\bar j}.
\]
Then
\[
\operatorname{Der}(N)
=
\operatorname{Der}(N)_{\bar 0}
\oplus
\operatorname{Der}(N)_{\bar 1}.
\]
Moreover, for
\[
D\in\operatorname{Der}_{\alpha^k}(N)_{\bar i},
\qquad
D'\in\operatorname{Der}_{\alpha^{k'}}(N)_{\bar j},
\]
we have
\[
[D,D']
\in
\operatorname{Der}_{\alpha^{k+k'}}(N)_{\bar i+\bar j}.
\]
Hence the supercommutator
\[
[D,D']
=
D\circ D'
-
(-1)^{|D||D'|}D'\circ D
\]
defines a Lie superalgebra bracket on
\[
\operatorname{Der}(N)
=
\operatorname{Der}(N)_{\bar0}
\oplus
\operatorname{Der}(N)_{\bar1}.
\]
\end{lemma}

\begin{proof}
Since both $D$ and $D'$ commute with $\alpha$, we have
\[
[D,D']\circ\alpha
=
\alpha\circ[D,D'].
\]
Applying the $\alpha^k$-derivation identity for $D$ and the
$\alpha^{k'}$-derivation identity for $D'$ to
$[x_1,\ldots,x_n]$, and then subtracting the two resulting
expressions with the sign
$(-1)^{|D||D'|}$, the terms containing two derivations cancel
pairwise. The remaining terms give
\[
\begin{aligned}
&[D,D']([x_1,\ldots,x_n])\\
&=
\sum_{i=1}^{n}
(-1)^{(|D|+|D'|)(|x_1|+\cdots+|x_{i-1}|)}
\bigl[
\alpha^{k+k'}(x_1),\ldots,
\alpha^{k+k'}(x_{i-1}),\\
&\hspace{4.7cm}
[D,D'](x_i),
\alpha^{k+k'}(x_{i+1}),\ldots,
\alpha^{k+k'}(x_n)
\bigr].
\end{aligned}
\]
Therefore,
\(
[D,D']
\in
\operatorname{Der}_{\alpha^{k+k'}}(N).
\)
Moreover, if
\(
|D|=\bar i
\)
and
\(
|D'|= \bar j,
\)
then
\(
|[D,D']|=\bar i+\bar j.
\)
Hence,
\[
[\operatorname{Der}(N)_{\bar i},\operatorname{Der}(N)_{\bar j}]
\subseteq
\operatorname{Der}(N)_{\bar i+\bar j}.
\]
The super-skew-symmetry and the super-Jacobi identity follow from
the corresponding properties of the supercommutator in
$\operatorname{Lin}(N)$. Therefore,
\[
\operatorname{Der}(N)
=
\operatorname{Der}(N)_{\bar0}
\oplus
\operatorname{Der}(N)_{\bar1}
\]
is a Lie superalgebra.
\end{proof}

\begin{lemma}
The map \(\operatorname{ad}^{k}_{X}\) defined above is an
\(\alpha^{k+1}\)\,--derivation of \(N\), called an
\emph{inner \(\alpha^{k+1}\)\,--derivation}, and its parity satisfies
\(|\operatorname{ad}^{k}_{X}| = |X|\).
\end{lemma}

\begin{proof}
Let $k\geq 0$ and let \( X=(x_1,\ldots,x_{n-1})\in H(N)^{n-1} \) be such that $\alpha(X)=X$.
First, we show that $\operatorname{ad}^{k}_{X}$ commutes with $\alpha$. For any
$y\in H(N)$, using the multiplicativity of $N$ and the condition
$\alpha(x_i)=x_i$ for all $1\leq i\leq n-1$, we obtain
\[
\begin{aligned}
\alpha\bigl(\operatorname{ad}^{k}_{X}(y)\bigr)
&=
\alpha\bigl([x_1,\ldots,x_{n-1},\alpha^k(y)]\bigr)\\
&=
[\alpha(x_1),\ldots,\alpha(x_{n-1}),\alpha^{k+1}(y)]\\
&=
[x_1,\ldots,x_{n-1},\alpha^{k+1}(y)] =
\operatorname{ad}^{k}_{X}(\alpha(y)).
\end{aligned}
\]
Hence,
\(
\alpha\circ\operatorname{ad}^{k}_{X}
=
\operatorname{ad}^{k}_{X}\circ\alpha.
\)

Now, let $y_1,\ldots,y_n\in H(N)$. Since $N$ is multiplicative,
\[
\alpha^k([y_1,\ldots,y_n])
=
[\alpha^k(y_1),\ldots,\alpha^k(y_n)].
\]
Therefore,
\[
\begin{aligned}
\operatorname{ad}^{k}_{X}([y_1,\ldots,y_n])
&=
[x_1,\ldots,x_{n-1},
\alpha^k([y_1,\ldots,y_n])]\\
&=
[x_1,\ldots,x_{n-1},
[\alpha^k(y_1),\ldots,\alpha^k(y_n)]]\\
&=
[\alpha(x_1),\ldots,\alpha(x_{n-1}),
[\alpha^k(y_1),\ldots,\alpha^k(y_n)]].
\end{aligned}
\]
Applying the super-Hom-Nambu identity \eqref{shni}, we obtain
\[
\begin{aligned}
&\operatorname{ad}^{k}_{X}([y_1,\ldots,y_n])\\
&=
\sum_{i=1}^{n}
(-1)^{|X||Y|_{i-1}}
\bigl[
\alpha^{k+1}(y_1),\ldots,
\alpha^{k+1}(y_{i-1}),
[x_1,\ldots,x_{n-1},\alpha^k(y_i)],\\
&\hspace{7cm}
\alpha^{k+1}(y_{i+1}),\ldots,
\alpha^{k+1}(y_n)
\bigr]\\
&=
\sum_{i=1}^{n}
(-1)^{|X||Y|_{i-1}}
\bigl[
\alpha^{k+1}(y_1),\ldots,
\alpha^{k+1}(y_{i-1}),
\operatorname{ad}^{k}_{X}(y_i),\\
&\hspace{7cm}
\alpha^{k+1}(y_{i+1}),\ldots,
\alpha^{k+1}(y_n)
\bigr],
\end{aligned}
\]
where
\(
|Y|_{i-1}=\sum\limits_{j=1}^{i-1}|y_j|.
\)
Thus, $\operatorname{ad}^{k}_{X}$ satisfies the
$\alpha^{k+1}$-derivation identity.

Finally, because the $n$-ary bracket and $\alpha$ are even, for every
$y\in H(N)$,
\[
\begin{aligned}
\left|\operatorname{ad}^{k}_{X}(y)\right|
&=
\left|[x_1,\ldots,x_{n-1},\alpha^k(y)]\right|\\
&=
|x_1|+\cdots+|x_{n-1}|+|\alpha^k(y)|
= |X|+|y|.
\end{aligned}
\]
Consequently,
\(
\left|\operatorname{ad}^{k}_{X}\right|=|X|.
\)
Therefore, $\operatorname{ad}^{k}_{X}$ is an
$\alpha^{k+1}$-derivation of $N$ of parity $|X|$.
\end{proof}

\section{Completeness of \texorpdfstring{\(n\)}{n}-ary Hom-Nambu superalgebras} \label{sec3-ComplnaryHomNambuSuperalg}
Let us consider the \(n\)-ary setting. Because an \(n\)-ary Hom-Nambu bracket is not necessarily totally super-skew-symmetric, the position occupied by a central element or an ideal element may be relevant. Therefore, we introduce positional centers and require Hom-ideals to be stable in every argument. These notions are reduced to the usual ones when the bracket is totally super-skew-symmetric.

We also introduce and examine several fundamental structural notions of $n$-ary Hom-Nambu superalgebras, including Hom-subalgebras, Hom-ideals, centers, and centralizers.
These concepts, extending the classical ideas from Lie and Hom-Lie theory \cite{farhangdoost2023} to the $n$-ary superalgebraic framework, will serve as essential tools for analyzing the completeness, derivations, and decomposition properties of $n$-ary Hom-Nambu superalgebras.
\begin{definition}[\cite{farhangdoost2023}]
	Let $(g=g_{\bar0}\oplus g_{\bar1},\,[\cdot\!,\!\cdot],\alpha)$ be a Hom-superalgebra (i.e. a $\mathbb Z_2$-graded linear space equipped with an even linear map $\alpha$ and a bilinear bracket $[\,,\,]$ satisfying
\(
[g_{\bar i},g_{\bar j}]\subseteq g_{\bar i+\bar j}
\)
for all \(\bar i,\bar j\in\mathbb Z_2\)).
	A \emph{graded Hom-subalgebra} of $g$ is a $\mathbb Z_2$-graded linear subspace
	\[
	I=(I\cap g_{\bar0})\oplus (I\cap g_{\bar1})\subseteq g
	\]
	such that
	\(
	\alpha(I)\subseteq I\)
    and
    \([I,I]\subseteq I.
	\)
\end{definition}

\begin{remark}
For \(k\geq 0\), we denote by
\(\operatorname{Der}_{\alpha^{k}}(N)\) the space of all
\(\alpha^{k}\)-derivations of \(N\). We denote by
\(\operatorname{Inn}_{\alpha^{k}}(N)\) the subspace of
\(\operatorname{Der}_{\alpha^{k+1}}(N)\) generated by all inner
\(\alpha^{k+1}\)-derivations of \(N\). We define the spaces of
derivations and inner derivations by
\[
\operatorname{Der}(N)
=
\bigoplus_{k\geq 0}
\operatorname{Der}_{\alpha^{k}}(N),
\qquad
\operatorname{Inn}(N)
=
\bigoplus_{k\geq 0}
\operatorname{Inn}_{\alpha^{k}}(N),
\]
where these direct sums are understood as graded external direct sums.
\end{remark}

\begin{definition}[\cite{farhangdoost2023}]
	Let $(\mathfrak{g},[\cdot,\cdot],\alpha)$ be a Hom-superalgebra and $I\subseteq \mathfrak{g}$ be a graded Hom-subalgebra.
	We say that $I$ is a \emph{left Hom-ideal} of $\mathfrak{g}$ if $[\mathfrak{g},I]\subseteq I$ and that it is a \emph{right Hom-ideal} if $[I,\mathfrak{g}]\subseteq I$.
	If both conditions hold, then $I$ is called a \emph{Hom-ideal} (or two-sided Hom-ideal), and we write $I\lhd \mathfrak{g}$.
\end{definition}

\begin{definition}[\cite{farhangdoost2023}]
	Let $(\mathfrak{g},[\cdot,\cdot],\alpha)$ be a Hom-Lie superalgebra.
	The center of $\mathfrak{g}$ is
	\(
	C(\mathfrak{g}) \;=\;\{\,x\in \mathfrak{g} \;\mid\; [x,\mathfrak{g}]=0\,\}.
	\)
	If $I \subseteq \mathfrak{g}$ is a Hom-ideal, the \emph{centralizer} of $I$ in $\mathfrak{g}$ is
	\(
	C_{\mathfrak{g}}(I) \;=\;\{\,x\in \mathfrak{g} \;\mid\; [x,I]=0\,\}.
	\)
\end{definition}

\begin{definition}[\cite{farhangdoost2023}] \label{def:complete}
A multiplicative Hom-Lie superalgebra $(\mathfrak{g},[\cdot,\cdot],\alpha)$ is called complete if it satisfies the following conditions, for each $k\geq 0$\textup{:}
\[
C(\mathfrak{g}) = 0, \qquad \operatorname{Der}_{\alpha^{k+1}}(\mathfrak{g}) = \operatorname{ad}_{\alpha^k
}(\mathfrak{g}).
\]
\end{definition}

\begin{proposition}
For every $\lambda\in\mathbb{R}^{*}$, the Hom-Lie superalgebra
\[
\mathfrak{osp}(1,2)_\lambda
=
\bigl(
\mathfrak{osp}(1,2),
[\cdot,\cdot]_{\alpha_\lambda},
\alpha_\lambda
\bigr)
\]
defined in Example \ref{osposp} is complete.
\end{proposition}

\begin{proof}
Let
\(
\mathfrak{g}=\mathfrak{osp}(1,2)
\)
and denote by $[\cdot,\cdot]$ the ordinary Lie superbracket. Recall that
\(
[x,y]_{\alpha_\lambda}=\alpha_\lambda([x,y])
\)
for all $x,y\in\mathfrak{g}$. Since $\lambda\neq0$, the twisting map $\alpha_\lambda$ is an automorphism of $\mathfrak{g}$.

Suppose that $x\in C(\mathfrak{osp}(1,2)_\lambda)$. Then
\(
0=[x,y]_{\alpha_\lambda}
=\alpha_\lambda([x,y])
\)
for every $y\in\mathfrak{g}$. Since $\alpha_\lambda$ is injective, we have $[x,y]=0$ for every $y\in\mathfrak{g}$. Therefore,
\(
x\in C(\mathfrak{osp}(1,2))=\{0\},
\)
and hence
\(
C(\mathfrak{osp}(1,2)_\lambda)=\{0\}.
\)

Let $k\geq 0$, and let
\(
D\in
\operatorname{Der}_{\alpha_\lambda^{k+1}}
(\mathfrak{osp}(1,2)_\lambda)
\)
be homogeneous.
Since $D$ commutes with $\alpha_\lambda$, the linear map
\(
E=\alpha_\lambda^{-(k+1)}\circ D
\)
commutes with $\alpha_\lambda$. Using
$[\cdot,\cdot]_{\alpha_\lambda}
=\alpha_\lambda\circ[\cdot,\cdot]$
in the $\alpha_\lambda^{k+1}$-derivation identity, one obtains
\[
E([x,y])
=
[E(x),y]
+
(-1)^{|E||x|}[x,E(y)]
\]
for all homogeneous $x,y\in\mathfrak{g}$. Thus, $E$ is a homogeneous derivation of the Lie superalgebra
$\mathfrak{osp}(1,2)$, of parity $|E|=|D|$, where by a homogeneous
derivation we mean a homogeneous linear map satisfying
\[
E([x,y])
=
[E(x),y]
+
(-1)^{|E||x|}[x,E(y)]
\]
for all homogeneous $x,y\in\mathfrak{g}$.

It is well known that every homogeneous derivation of the simple Lie
superalgebra $\mathfrak{osp}(1,2)$ is inner
\cite{SerganovaVaintrob2026}, that is,
\(
\operatorname{Der}(\mathfrak{osp}(1,2))
=
\operatorname{Inn}(\mathfrak{osp}(1,2)).
\)
Hence there exists a homogeneous element \(z\in\mathfrak{g}\) such that
\(
E=\operatorname{ad}_z.
\)
The equality
\(
E\circ\alpha_\lambda
=
\alpha_\lambda\circ E
\)
implies
\(
[z-\alpha_\lambda(z),\alpha_\lambda(x)]=0
\)
for all $x\in\mathfrak{g}$. Since $\alpha_\lambda$ is surjective and
$C(\mathfrak{g})=\{0\}$, it follows that
\(
\alpha_\lambda(z)=z.
\)

Consequently, for every $x\in\mathfrak{g}$,
\[
D(x)=
\alpha_\lambda^{k+1}([z,x])=
[z,\alpha_\lambda^{k+1}(x)]=
[z,\alpha_\lambda^k(x)]_{\alpha_\lambda}=
\operatorname{ad}_{\alpha_\lambda^k}^{z}(x).
\]
Therefore, every $\alpha_\lambda^{k+1}$-derivation is inner. Hence,
\[
\operatorname{Der}_{\alpha_\lambda^{k+1}}
(\mathfrak{osp}(1,2)_\lambda)
=
\operatorname{ad}_{\alpha_\lambda^k}
(\mathfrak{osp}(1,2)_\lambda).
\]
Since $k\geq 0$ was arbitrary,
$\mathfrak{osp}(1,2)_\lambda$ is complete for every
$\lambda\in\mathbb{R}^{*}$.
\end{proof}

\begin{definition}
Let \( (N,[\cdot,\ldots,\cdot],\alpha) \) be a multiplicative \(n\)-ary Hom-Nambu superalgebra. A
Hom-subalgebra of \(N\) is a graded linear subspace
\(
I=(I\cap N_{\bar 0})\oplus(I\cap N_{\bar 1})
\)
such that
\(
\alpha(I)\subseteq I
\)
and
\(
[I,\ldots,I]\subseteq I.
\)
\end{definition}

\begin{definition}
Let
\(
(N,[\cdot,\ldots,\cdot],\alpha)
\)
be a multiplicative \(n\)-ary Hom-Nambu superalgebra, and let
\(I\) be a Hom-subalgebra of \(N\). We call \(I\) a Hom-ideal of
\(N\) if
\(
[N,\ldots,N,
\underset{i\text{-th position}}{I},
N,\ldots,N]\subseteq I
\)
for each \(1\leq i\leq n\). In this case, we write
\(
I\triangleleft N.
\)
\end{definition}
\begin{definition}
Let
\(
(N,[\cdot,\ldots,\cdot],\alpha)
\)
be a multiplicative \(n\)-ary Hom-Nambu superalgebra. For
\(1\leq i\leq n\), the \(i\)-th positional center of \(N\) is
defined by
\[
C_i(N)=
\left\{
z\in N\ \middle|\
[x_1,\ldots,x_{i-1},z,x_{i+1},\ldots,x_n]=0
\text{ for all }x_j\in N
\right\}.
\]
The center of \(N\) is
\(
C(N)=\bigcap\limits_{i=1}^{n}C_i(N).
\)
More generally, if \(I\) is a Hom-ideal of \(N\), we define
\[
C_i(I)=
\left\{
z\in N\ \middle|\
[x_1,\ldots,x_{i-1},z,x_{i+1},\ldots,x_n]=0
\text{ for all }x_j\in I
\right\},
\]
and set
\(
C_N(I)=\bigcap\limits_{i=1}^{n}C_i(I).
\)
\end{definition}


We present several structural results concerning Hom-ideals, centers, and centralizers in $n$-ary Hom-Lie superalgebras. These propositions clarify how the Hom-ideal property can be expressed equivalently, how the center behaves under multiplicativity, and how completeness arises under specific conditions.

\begin{proposition}
Let
\(
(N,[\cdot,\ldots,\cdot],\alpha)
\)
be an \(n\)-Hom-Lie superalgebra.

\begin{enumerate}
\item
For every \(1\leq i,j\leq n\),
\(
C_i(N)=C_j(N).
\)

\item
Let \(I\) be a graded \(\alpha\)-stable subspace of \(N\). If
\[
[N,\ldots,N,
\underset{i\text{-th position}}{I},
N,\ldots,N]\subseteq I
\]
for one position \(i\), then the same inclusion holds for every
position.
\end{enumerate}
\end{proposition}
\begin{proof}
Both assertions follow from the total super-skew-symmetry. An element may be moved from one position to any other position by a sequence of adjacent transpositions. Each transposition changes the bracket only by a nonzero Koszul sign.
\end{proof}

This proposition shows that in Hom-Lie superalgebras, the Hom-ideal condition is invariant under cyclic permutation of the arguments in the bracket.
Thus, it does not matter whether the element of \(I\) appears in the first or the last position.

\begin{proposition}
Let \( (N,[\cdot,\ldots,\cdot],\alpha) \) be a multiplicative \(n\)-ary Hom-Nambu superalgebra with a surjective twisting map \(\alpha\). Then the center
	\(
	C(N)=\bigcap\limits_{i=1}^{n}C_i(N)
	\)
	is an abelian Hom-ideal of \(N\).
\end{proposition}

\begin{proof}
	First, we prove that \(C(N)\) is stable under the twisting map.
	Let \(x\in C(N)\) and fix \(1\leq i\leq n\). For arbitrary
	elements
	\[
	y_1,\ldots,y_{i-1},y_{i+1},\ldots,y_n\in N,
	\]
	the surjectivity of \(\alpha\) implies that there exist elements
	\[
	b_1,\ldots,b_{i-1},b_{i+1},\ldots,b_n\in N
	\]
	such that
	\(
	y_j=\alpha(b_j),\ j\neq i.
	\)
	Using multiplicativity, we obtain
	\[
	\begin{aligned}
	&[y_1,\ldots,y_{i-1},\alpha(x),
	y_{i+1},\ldots,y_n]
	\\
	&\quad=
	[\alpha(b_1),\ldots,\alpha(b_{i-1}),\alpha(x),
	\alpha(b_{i+1}),\ldots,\alpha(b_n)]
	\\
	&\quad=
	\alpha\bigl(
	[b_1,\ldots,b_{i-1},x,b_{i+1},\ldots,b_n]
	\bigr).
	\end{aligned}
	\]
	Since \(x\in C(N)\subseteq C_i(N)\), we have
	\[
	[b_1,\ldots,b_{i-1},x,b_{i+1},\ldots,b_n]=0.
	\]
	Hence,
	\(
	[y_1,\ldots,y_{i-1},\alpha(x),
	y_{i+1},\ldots,y_n]=0.
	\)
	Since \(i\) was arbitrary,
	\[
	\alpha(x)\in\bigcap_{i=1}^{n}C_i(N)=C(N).
	\]
	Therefore,
	\(
	\alpha(C(N))\subseteq C(N).
	\)

	Next, let \(x\in C(N)\) and fix \(1\leq i\leq n\). By the
	definition of the center, for arbitrary elements
	\(x_j\in N\), \(j\neq i\), we have
	\[
	[x_1,\ldots,x_{i-1},x,x_{i+1},\ldots,x_n]=0.
	\]
	Consequently,
	\(
	[N,\ldots,N,
	\underset{i\text{-th position}}{C(N)},
	N,\ldots,N]
	\subseteq \{0\}\subseteq C(N)
	\)
	for each \(1\leq i\leq n\). Hence, \(C(N)\) is a Hom-ideal of
	\(N\).

	Finally, because every element of \(C(N)\) is central in every position,
	\[
	[C(N),\ldots,C(N)]=\{0\}.
	\]
	Therefore, \(C(N)\) is an abelian Hom-ideal of \(N\).
\end{proof}

The above result shows that the center behaves well under multiplicative and surjective twisting,
retaining algebraic stability and internal commutativity.

\begin{proposition}\label{surcom}
	Let
	\(
	(N,[\cdot,\ldots,\cdot],\alpha)
	\)
	be a multiplicative \(n\)-ary Hom-Nambu superalgebra such that
	\(
	\alpha=0.
	\)
	If the bracket
	\(
	[\cdot,\ldots,\cdot]\colon N^n\longrightarrow N
	\)
	is surjective, then
	\(
	\operatorname{Der}_{\alpha^{k+1}}(N)
	=
	\operatorname{Inn}_{\alpha^k}(N)
	=
	\{0\}
	\)
    for every integer \(k\geq0\).
	 Consequently, \(N\) is complete
	if and only if
	\(
	C(N)=\{0\}.
	\)
\end{proposition}

\begin{proof}
	Let \(k\geq0\), and let
	\(
	D\in\operatorname{Der}_{\alpha^{k+1}}(N)
	\)
	be homogeneous. For arbitrary homogeneous elements
	\(x_1,\ldots,x_n\in N\), the
	\(\alpha^{k+1}\)-derivation identity gives
	\[
	\begin{aligned}
	&D([x_1,\ldots,x_n])
	=
	\sum_{i=1}^{n}
	(-1)^{|D|(|x_1|+\cdots+|x_{i-1}|)}
	\\
	&\quad\times
	[\alpha^{k+1}(x_1),\ldots,
	\alpha^{k+1}(x_{i-1}),D(x_i),
	\alpha^{k+1}(x_{i+1}),\ldots,
	\alpha^{k+1}(x_n)].
	\end{aligned}
	\]
	Since \(\alpha=0\) and \(k+1\geq1\), we have
	\(
	\alpha^{k+1}(x_j)=0
	\)
	for each \(1\leq j\leq n\). Hence, each term in the preceding sum contains at least one zero argument. By multilinearity of
	the bracket,
	\(
	D([x_1,\ldots,x_n])=0.
	\)

	Since the bracket is surjective, for every \(x\in N\) there
	exist \(x_1,\ldots,x_n\in N\) such that
	\(
	x=[x_1,\ldots,x_n].
	\)
	Therefore,
	\(
	D(x)=D([x_1,\ldots,x_n])=0.
	\)
	It follows that
	\(
	\operatorname{Der}_{\alpha^{k+1}}(N)=\{0\}.
	\)

	Next, we consider the inner derivations. Let
	\[
	X=(x_1,\ldots,x_{n-1})\in H(N)^{n-1}
	\]
	satisfy
	\(
	\alpha(X)=X.
	\)
	This condition means that
	\(
	\alpha(x_i)=x_i,
	\)
    for each
    \(1\leq i\leq n-1.
	\)
	Since \(\alpha=0\), we obtain
	\(
	x_i=0
	\)
    for each \(1\leq i\leq n-1.\)
	Thus,
	\(
	X=(0,\ldots,0).
	\)
	Consequently, for every \(y\in N\),
	\[
	\operatorname{ad}_X^k(y)
	=
	[x_1,\ldots,x_{n-1},\alpha^k(y)]
	=0.
	\]
	Hence every inner \(\alpha^{k+1}\)-derivation vanishes, and
	therefore
	\(
	\operatorname{Inn}_{\alpha^k}(N)=\{0\}.
	\)
Thus, we have proved
\(\operatorname{Der}_{\alpha^{k+1}}(N)=\operatorname{Inn}_{\alpha^k}(N)=\{0\}\)
for every \(k\geq0\).
Consequently, the derivation condition in the definition of completeness is automatically satisfied. Therefore, \(N\) is complete if and only if its center is trivial, that is, \( C(N)=\{0\}.\)
\end{proof}

The above proposition shows that, when the twisting map is zero and
the bracket is surjective, the derivation condition in the definition
of completeness is automatically satisfied. In this case, completeness is equivalent to the triviality of the center.

The following example shows that the preceding criterion produces
complete \(n\)-ary Hom-Nambu superalgebras that are not
\(n\)-Hom-Lie superalgebras.

\begin{example}\label{ex:complete-nonskew-Hom-Nambu}
Let \(n\geq2\), and let
\[
N=N_{\bar 0}
=
\operatorname{span}_{\mathbb{K}}\{e_1,e_2\},
\qquad
N_{\bar 1}=\{0\}.
\]
Define
\(
\alpha=0
\)
and an \(n\)-linear bracket on the basis elements by
\[
[e_1,\ldots,e_1]=e_1,
\qquad
[e_2,\ldots,e_2]=e_2,
\]
whereas every basic bracket containing both \(e_1\) and \(e_2\) is zero.
The bracket is extended to \(N\) using multilinearity.
First, we verify that
\(
\bigl(N,[\cdot,\ldots,\cdot],\alpha\bigr)
\)
is a multiplicative \(n\)-ary Hom-Nambu superalgebra. Because \(N\) is purely even, the bracket is even. Moreover, for all
\(x_1,\ldots,x_n\in N\),
\[
\alpha([x_1,\ldots,x_n])
=
0
=
[\alpha(x_1),\ldots,\alpha(x_n)].
\]
Thus, multiplicativity holds.
The super-Hom-Nambu identity is also satisfied. Indeed, because \(\alpha=0\) and \(n\geq2\), every bracket on either side of the super-Hom-Nambu identity contains at least one zero argument.
Therefore, by multilinearity, both sides vanish.
Next, we show that the bracket is surjective. Let
\(
x=ae_1+be_2\in N.
\)
Then
\[
\left[
\underbrace{e_1+e_2,\ldots,e_1+e_2}_{n-1\text{ arguments}},
ae_1+be_2
\right]
=
ae_1+be_2
=
x.
\]
Indeed, all mixed terms in the multilinear expansion vanish, and only
\[
a[e_1,\ldots,e_1]=ae_1, \quad
b[e_2,\ldots,e_2]=be_2
\]
remain.
We now determine the positional centers. Fix
\(1\leq i\leq n\), and let
\[
z=ae_1+be_2\in C_i(N).
\]
Placing \(z\) in the \(i\)-th position and \(e_1\) in every other
position gives
\[
[e_1,\ldots,e_1,
\underset{i\text{-th position}}{z},
e_1,\ldots,e_1]
=
ae_1.
\]
Because \(z\in C_i(N)\), this bracket must vanish, and hence \(a=0\).
Similarly,
\[
[e_2,\ldots,e_2,
\underset{i\text{-th position}}{z},
e_2,\ldots,e_2]
=
be_2,
\]
and therefore \(b=0\). Thus
\(
C_i(N)=\{0\}
\)
for each \(1\leq i\leq n\). Consequently,
\[
C(N)
=
\bigcap_{i=1}^{n}C_i(N)
=
\{0\}.
\]
Since the bracket is surjective and \(\alpha=0\),
Proposition~\ref{surcom} implies that
\[
\operatorname{Der}_{\alpha^{k+1}}(N)
=
\operatorname{Inn}_{\alpha^k}(N)
=
\{0\}
\]
for every \(k\geq0\). Therefore,
\(
\bigl(N,[\cdot,\ldots,\cdot],\alpha\bigr)
\)
is complete.
Finally, this algebra is not an \(n\)-Hom-Lie superalgebra. Indeed,
total super-skew-symmetry would imply
\(
[e_1,e_1,x_3,\ldots,x_n]=0.
\)
However,
\[
[e_1,\ldots,e_1]=e_1\neq0.
\]
Thus, the bracket is not totally super-skew-symmetric.
\end{example}

\begin{remark}
Example~\ref{ex:complete-nonskew-Hom-Nambu} shows that completeness for
\(n\)-ary Hom-Nambu superalgebras is genuinely more general than completeness in the \(n\)-Hom-Lie category.
\end{remark}


\begin{proposition}
	Let
	\(
	(N,[\cdot,\ldots,\cdot],\alpha)
	\)
	be a multiplicative \(n\)-ary Hom-Nambu superalgebra, and let \(I\) be a
	Hom-ideal of \(N\). Then the centralizer
	\(
	C_N(I)=\bigcap_{i=1}^{n}C_i(I)
	\)
	is a graded subspace of \(N\). More precisely,
	\[
	C_N(I)
	=
	\bigl(C_N(I)\cap N_{\bar{0}}\bigr)
	\oplus
	\bigl(C_N(I)\cap N_{\bar{1}}\bigr).
	\]
\end{proposition}

\begin{proof}
	Let
	\(
	x=x_{\bar{0}}+x_{\bar{1}}\in C_N(I),
	\)
	where
	\(
	x_{\bar{0}}\in N_{\bar{0}}
	\)
    and
    \(
	x_{\bar{1}}\in N_{\bar{1}}.
	\)
	We prove that both \(x_{\bar{0}}\) and \(x_{\bar{1}}\) belong to
	\(C_N(I)\).

	Fix \(1\leq i\leq n\), and let \( y_1,\ldots,y_{i-1},y_{i+1},\ldots,y_n\in I \) be homogeneous elements. Since
	\(
	x\in C_N(I)\subseteq C_i(I),
	\)
	we have
	\(
	[y_1,\ldots,y_{i-1},x,
	y_{i+1},\ldots,y_n]=0.
	\)
	By linearity of the bracket,
	\[
	\begin{aligned}
	0
	&=
	[y_1,\ldots,y_{i-1},x_{\bar{0}}+x_{\bar{1}},
	y_{i+1},\ldots,y_n]
	\\
	&=
	[y_1,\ldots,y_{i-1},x_{\bar{0}},
	y_{i+1},\ldots,y_n]
	\\
	&\quad+
	[y_1,\ldots,y_{i-1},x_{\bar{1}},
	y_{i+1},\ldots,y_n].
	\end{aligned}
	\]
Set
\(
	\bar \varepsilon
	=
	|y_1|+\cdots+|y_{i-1}|
	+|y_{i+1}|+\cdots+|y_n|
	\in\mathbb Z_2.
	\)
	Since the bracket is even,
	\(
	[y_1,\ldots,y_{i-1},x_{\bar{0}},
	y_{i+1},\ldots,y_n]
	\in N_{\bar \varepsilon},
	\)
	whereas
	\[
	[y_1,\ldots,y_{i-1},x_{\bar{1}},
	y_{i+1},\ldots,y_n]
	\in N_{\bar{1}+\bar \varepsilon}.
	\]
	These two terms belong to distinct homogeneous components of \(N\). Therefore, each term must vanish,
	\[
	[y_1,\ldots,y_{i-1},x_{\bar{0}},
	y_{i+1},\ldots,y_n]=0,\quad
	[y_1,\ldots,y_{i-1},x_{\bar{1}},
	y_{i+1},\ldots,y_n]=0.
	\]

	By multilinearity, these equalities hold for arbitrary, not
	necessarily homogeneous, elements
	\(
	y_j\in I,\qquad j\neq i.
	\)
	Hence,
	\(
	x_{\bar{0}},x_{\bar{1}}\in C_i(I).
	\)
	Since the position \(i\) was arbitrary, we obtain
	\(
	x_{\bar{0}},x_{\bar{1}}
	\in\bigcap\limits_{i=1}^{n}C_i(I)
	=C_N(I).
	\)
	Thus,
	\[
	C_N(I)
	\subseteq
	\bigl(C_N(I)\cap N_{\bar{0}}\bigr)
	+
	\bigl(C_N(I)\cap N_{\bar{1}}\bigr).
	\]

	The reverse inclusion is immediate because \(C_N(I)\) is a
	linear subspace of \(N\). Moreover,
	\(
	\bigl(C_N(I)\cap N_{\bar{0}}\bigr)
	\cap
	\bigl(C_N(I)\cap N_{\bar{1}}\bigr)
	=\{0\},
	\)
	since
	\(
	N_{\bar{0}}\cap N_{\bar{1}}=\{0\}.
	\)
	Therefore,
	\(
	C_N(I)
	=
	\bigl(C_N(I)\cap N_{\bar{0}}\bigr)
	\oplus
	\bigl(C_N(I)\cap N_{\bar{1}}\bigr).
	\)
\end{proof}
The above result shows that the centralizer inherits the natural
$\mathbb{Z}_2$-grading from $N$.


We now introduce the notion of completeness for multiplicative $n$-ary Hom-Nambu superalgebras and provide an example.
Completeness expresses the idea that the center is trivial and that all
$\alpha^{k+1}$-derivations are inner for every $k\geq 0$, generalizing the classical Lie algebra concept to the Hom and super contexts.

\begin{definition}\label{def:complete-nary-Hom-Nambu}
Let
\(
(N,[\cdot,\ldots,\cdot],\alpha)
\)
be a multiplicative \(n\)-ary Hom-Nambu superalgebra. We say that
\(N\) is complete if
\(
C(N)=\{0\}
\)
and
\[
\operatorname{Der}_{\alpha^{k+1}}(N)
=
\operatorname{Inn}_{\alpha^k}(N)
\]
for every integer \(k\geq0.\)

If \(N\) is an \(n\)-Hom-Lie superalgebra, total
super-skew-symmetry implies
\[
C_1(N)=\cdots=C_n(N).
\]
Consequently, the preceding definition coincides with the usual
definition of completeness for multiplicative \(n\)-Hom-Lie
superalgebras.
\end{definition}

In other words, a complete multiplicative $n$-ary Hom-Nambu
superalgebra has no nonzero central elements, and every
$\alpha^{k+1}$-derivation is inner for each integer $k\geq0$.

\begin{example}\label{ex:standard-simple-n-lie}
Let
\(
\mathfrak{g}=\mathfrak{g}_{\bar 0}
\)
be a purely even linear space with basis
\[
\{e_1,\ldots,e_{n+1}\},
\qquad n\geq 2.
\]
Let
\(
\alpha=\operatorname{id}_{\mathfrak{g}},
\)
and define the $n$-linear bracket by
\[
[e_1,\ldots,\widehat{e_i},\ldots,e_{n+1}]
=
(-1)^i e_i,
\qquad
1\leq i\leq n+1,
\]
where $\widehat{e_i}$ denotes the omission of $e_i$. The remaining
brackets are determined by skew-symmetry and multilinearity. Then
\(
(\mathfrak{g},[\cdot,\ldots,\cdot],\alpha)
\)
is a multiplicative $n$-ary Hom-Lie superalgebra.
\end{example}

\begin{proposition}\label{prop:standard-n-lie-complete}
Let $\mathfrak{g}$ be the $(n+1)$-dimensional $n$-Lie algebra
defined in Example~\ref{ex:standard-simple-n-lie}, where
$n\geq 2$, and let
\(
\alpha=\operatorname{id}_{\mathfrak{g}}.
\)
Then
\(
(\mathfrak{g},[\cdot,\ldots,\cdot],\alpha)
\)
is complete.
\end{proposition}

\begin{proof}
The $n$-Lie algebra $\mathfrak{g}$ is the standard simple
$(n+1)$-dimensional $n$-Lie algebra $A_{n+1}$
\cite{filippov1985,deazcarraga2010}.
Since the center $C(\mathfrak{g})$ is an ideal of
$\mathfrak{g}$ and $\mathfrak{g}$ is simple and nonabelian, we have
\(
C(\mathfrak{g})=\{0\}.
\)
Moreover, every derivation of a simple $n$-Lie algebra is inner \cite{deazcarraga2010}. Hence,
\(
\operatorname{Der}(\mathfrak{g})
=
\operatorname{Inn}(\mathfrak{g}).
\)
In fact,
\(
\operatorname{Der}(\mathfrak{g})
=
\operatorname{Inn}(\mathfrak{g})
\cong
\mathfrak{so}(n+1)
\)
by \cite{deazcarraga2010}.

Finally, since
\(
\alpha=\operatorname{id}_{\mathfrak{g}},
\)
for every $k\geq0$ we have
\(
\operatorname{Der}_{\alpha^{k+1}}(\mathfrak{g})
=
\operatorname{Der}(\mathfrak{g})
\)
and
\(
\operatorname{Inn}_{\alpha^k}(\mathfrak{g})
=
\operatorname{Inn}(\mathfrak{g}).
\)
Consequently,
\(
\operatorname{Der}_{\alpha^{k+1}}(\mathfrak{g})
=
\operatorname{Inn}_{\alpha^k}(\mathfrak{g})
\)
for every $k\geq0$. Therefore, $\mathfrak{g}$ is complete.
\end{proof}

This classical construction provides a concrete example of a complete
$n$-Lie algebra (and hence a complete multiplicative $n$-Hom-Lie algebra when equipped with the identity map).
This illustrates that completeness can occur even in low-dimensional and fully symmetric structures.

\begin{lemma}\label{lem:twisting-nary-Hom-Nambu}
	Let
	\(
	(N,[\cdot,\ldots,\cdot],\alpha)
	\)
	be a multiplicative \(n\)-ary Hom-Nambu superalgebra, and let
	\(
	\beta:N\longrightarrow N
	\)
	be an endomorphism. Define
	\[
	[x_1,\ldots,x_n]_{\beta}
	=
	\beta([x_1,\ldots,x_n]),
	\qquad
	\widetilde{\alpha}
	=
	\beta\circ\alpha.
	\]
	Then
	\(
	\bigl(
	N,
	[\cdot,\ldots,\cdot]_{\beta},
	\widetilde{\alpha}
	\bigr)
	\)
	is a multiplicative \(n\)-ary Hom-Nambu superalgebra.

	If the original bracket is totally super-skew-symmetric, then the twisted structure is a multiplicative \(n\)-Hom-Lie superalgebra.
\end{lemma}

\begin{proof}
	Since \(\beta\) is an endomorphism, it is even and satisfies
	\[
	\beta([x_1,\ldots,x_n])
	=
	[\beta(x_1),\ldots,\beta(x_n)]
	\]
	and
	\(
	\beta\circ\alpha
	=
	\alpha\circ\beta.
	\)

	Let
	\(
	X=(x_1,\ldots,x_{n-1})\in H(N)^{n-1}
	,\
	Y=(y_1,\ldots,y_n)\in H(N)^n.
	\)
	For \(1\leq i\leq n\), let
	\(
	T_i
	={}
	[\alpha(y_1),\ldots,\alpha(y_{i-1}),
	[x_1,\ldots,x_{n-1},y_i],
	\
	\alpha(y_{i+1}),\ldots,\alpha(y_n)].
	\)
Using the definition of the twisted bracket and the fact that
	\(\beta\) preserves the original bracket, we have
	\[
	\begin{aligned}
	&[\widetilde{\alpha}(x_1),\ldots,
	\widetilde{\alpha}(x_{n-1}),
	[y_1,\ldots,y_n]_{\beta}]_{\beta}
	\\
	&\qquad=
	\beta^2
	\bigl(
	[\alpha(x_1),\ldots,\alpha(x_{n-1}),
	[y_1,\ldots,y_n]]
	\bigr).
	\end{aligned}
	\]
	By the super-Hom-Nambu identity,
	\[
	[\widetilde{\alpha}(x_1),\ldots,
	\widetilde{\alpha}(x_{n-1}),
	[y_1,\ldots,y_n]_{\beta}]_{\beta}
	=
	\sum_{i=1}^{n}
	(-1)^{|X||Y|_{i-1}}
	\beta^2(T_i).
	\]
Moreover,
	\[
	\beta^2(T_i)
	={}
	[\widetilde{\alpha}(y_1),\ldots,
	\widetilde{\alpha}(y_{i-1}),
	[x_1,\ldots,x_{n-1},y_i]_{\beta},
	\widetilde{\alpha}(y_{i+1}),\ldots,
	\widetilde{\alpha}(y_n)]_{\beta}.
	\]
	Therefore, the twisted bracket satisfies the
	super-Hom-Nambu identity.

	Next, we verify the multiplicativity. For arbitrary
	\(x_1,\ldots,x_n\in N\),
	\[
	\begin{aligned}
	\widetilde{\alpha}
	([x_1,\ldots,x_n]_{\beta})
	&=
	(\beta\circ\alpha)
	\bigl(\beta([x_1,\ldots,x_n])\bigr)
	\\
	&=
	\beta^2
	\bigl(\alpha([x_1,\ldots,x_n])\bigr)
	\\
	&=
	\beta^2
	\bigl(
	[\alpha(x_1),\ldots,\alpha(x_n)]
	\bigr)
	\\
	&=
	[\widetilde{\alpha}(x_1),\ldots,
	\widetilde{\alpha}(x_n)]_{\beta}.
	\end{aligned}
	\]
	Hence, the twisted structure is multiplicative.
Finally, suppose that the original bracket is totally
	super-skew-symmetric. For \(1\leq i\leq n-1\), we have
	\[
	[x_1,\ldots,x_i,x_{i+1},\ldots,x_n]_{\beta}
	=
	-(-1)^{|x_i||x_{i+1}|}
	[x_1,\ldots,x_{i+1},x_i,\ldots,x_n]_{\beta}.
	\]
	Thus, the twisted bracket is also totally
	super-skew-symmetric, and the resulting structure is a
	multiplicative \(n\)-Hom-Lie superalgebra.
\end{proof}


Next, we study the behavior of the center under twisting by powers
of the structure map~$\alpha$.

\begin{proposition}\label{prop:center-twist}
	Let
	\(
	(N,[\cdot,\ldots,\cdot],\alpha)
	\)
	be a multiplicative \(n\)-ary Hom-Nambu superalgebra with a
	bijective twisting map \(\alpha\). For every \(k\in\mathbb{Z}\),
	define
	\[
	[x_1,\ldots,x_n]_{\alpha^k}
	=
	\alpha^k([x_1,\ldots,x_n])
	\]
	and set
	\(
	N_{\alpha^k}
	=
	\bigl(
	N,[\cdot,\ldots,\cdot]_{\alpha^k},
	\alpha^{k+1}
	\bigr).
	\)
	Then \(N_{\alpha^k}\) is a multiplicative \(n\)-ary
	Hom-Nambu superalgebra. Moreover, for each
	\(1\leq i\leq n\),
	\[
	C_i(N_{\alpha^k})=C_i(N).
	\]
	Consequently,
	\(
	C(N_{\alpha^k})=C(N),
	\)
	and if \(N\) is complete, then
	\(
	C(N_{\alpha^k})=\{0\}.
	\)
\end{proposition}

\begin{proof}
	Since \(\alpha\) is bijective and multiplicative, every integral
	power
	\[
	\alpha^k\colon N\longrightarrow N,
	\qquad k\in\mathbb{Z},
	\]
	is an even bijective morphism of the \(n\)-ary bracket. Moreover,
	\(\alpha^k\) commutes with \(\alpha\). Therefore, by the preceding
	twisting result,
	\[
	N_{\alpha^k}
	=
	\bigl(
	N,[\cdot,\ldots,\cdot]_{\alpha^k},
	\alpha^{k+1}
	\bigr)
	\]
	is a multiplicative \(n\)-ary Hom-Nambu superalgebra.

	We now compare the positional centers. Fix
	\(
	1\leq i\leq n.
	\)
	Let \(z\in C_i(N_{\alpha^k})\). Then, for arbitrary elements
	\(
	x_1,\ldots,x_{i-1},x_{i+1},\ldots,x_n\in N,
	\)
	we have
	\[
	[x_1,\ldots,x_{i-1},z,
	x_{i+1},\ldots,x_n]_{\alpha^k}=0.
	\]
	By the definition of the twisted bracket,
	\[
	\alpha^k\bigl(
	[x_1,\ldots,x_{i-1},z,
	x_{i+1},\ldots,x_n]
	\bigr)=0.
	\]
	Since \(\alpha^k\) is injective, it follows that
	\[
	[x_1,\ldots,x_{i-1},z,
	x_{i+1},\ldots,x_n]=0.
	\]
	Therefore,
	\(
	z\in C_i(N),
	\)
	and hence
	\(
	C_i(N_{\alpha^k})\subseteq C_i(N).
	\)

	Conversely, let \(z\in C_i(N)\). Then for all \(x_j\in N\), \(j\neq i\),
	\[
	[x_1,\ldots,x_{i-1},z,
	x_{i+1},\ldots,x_n]=0.
	\]
	Consequently,
	\[
	\begin{aligned}
	&[x_1,\ldots,x_{i-1},z,
	x_{i+1},\ldots,x_n]_{\alpha^k}
	\\
	&\quad=
	\alpha^k\bigl(
	[x_1,\ldots,x_{i-1},z,
	x_{i+1},\ldots,x_n]
	\bigr) =0.
	\end{aligned}
	\]
	Thus,
	\(
	z\in C_i(N_{\alpha^k}),
	\)
	and therefore
	\(
	C_i(N)\subseteq C_i(N_{\alpha^k}).
	\)
	We conclude that
	\(
	C_i(N_{\alpha^k})=C_i(N)
	\)
	for each \(1\leq i\leq n\).

	Taking the intersection over all positions gives
	\[
	C(N_{\alpha^k})
	=
	\bigcap_{i=1}^{n}C_i(N_{\alpha^k})
	=
	\bigcap_{i=1}^{n}C_i(N)
	=
	C(N).
	\]
Finally, if \(N\) is complete, then
	\(
	C(N)=\{0\}.
	\)
	Hence,
	\(
	C(N_{\alpha^k})=C(N)=\{0\}.
	\)
\end{proof}
The above proposition shows that the triviality of the center is
preserved under twisting by powers of a bijective structure map.
Preservation of completeness would additionally require a comparison
of the corresponding derivation and inner derivation spaces.

\begin{theorem}\label{thm-multHomLiedirdecompideals}
	Let
	\(
	(\mathfrak{g},[\cdot,\ldots,\cdot],\alpha)
	\)
	be a multiplicative \(n\)-ary Hom-Lie superalgebra such that
	\(
	\mathfrak{g}=I\oplus J,
	\)
	where \(I\) and \(J\) are Hom-ideals of \(\mathfrak{g}\), and
	\(\alpha\) is bijective. Here, \(\alpha\) also denotes its
	restrictions to \(I\) and \(J\). Then, the following statements hold:
	\begin{enumerate}[label=\textup{\arabic*)},ref=\textup{\arabic*}]
		\item\label{item1-thm-multHomLiedirdecompideals}
		\(
		C(\mathfrak{g})=C(I)\oplus C(J).
		\)
		\item\label{item2-thm-multHomLiedirdecompideals}
		If \(C(\mathfrak{g})=\{0\}\), then, for every \(k\geq0\),
		\begin{align*}
		 \operatorname{Inn}_{\alpha^k}(\mathfrak{g})
		&=
		\operatorname{Inn}_{\alpha^k}(I)
		\oplus
		\operatorname{Inn}_{\alpha^k}(J),
		\\
		\operatorname{Der}_{\alpha^{k+1}}(\mathfrak{g})
		&=
		\operatorname{Der}_{\alpha^{k+1}}(I)
		\oplus
		\operatorname{Der}_{\alpha^{k+1}}(J).
		\end{align*}
		\item\label{item3-thm-multHomLiedirdecompideals}
		The \(n\)-Hom-Lie superalgebra \(\mathfrak{g}\) is complete if and
		only if both \(I\) and \(J\) are complete.
	\end{enumerate}
\end{theorem}

\begin{proof}
Because \(I\) and \(J\) are Hom-ideals of \(\mathfrak{g}\), every bracket containing at least one element of \(I\) and at least one element of \(J\) belongs to both \(I\) and \(J\). Hence,
	\(
	[I,J,\mathfrak{g},\ldots,\mathfrak{g}]
	\subseteq I\cap J=\{0\}.
	\)
Thus, every mixed bracket vanishes.

	We first show that the restrictions of \(\alpha\) to \(I\) and
	\(J\) are bijective. Their injectivity follows from the injectivity
	of \(\alpha\). Let \(y\in I\). Since \(\alpha\) is surjective,
	there exists \(x\in\mathfrak{g}\) such that
	\(
	\alpha(x)=y.
	\)
	Write
	\[
	x=x_I+x_J,
	\qquad x_I\in I,\quad x_J\in J.
	\]
	Since \(I\) and \(J\) are \(\alpha\)-stable,
	\(
	\alpha(x_I)\in I,
	\alpha(x_J)\in J.
	\)
	Moreover,
	\(
	y=\alpha(x_I)+\alpha(x_J)\in I.
	\)
	Because \(\mathfrak{g}=I\oplus J\), we obtain
	\(
	\alpha(x_J)=0.
	\)
	The injectivity of \(\alpha\) implies \(x_J=0\), and hence
	\(
	y=\alpha(x_I).
	\)
	Therefore, \(\left.\alpha\right|_I\) is surjective. The same
	argument proves that \(\left.\alpha\right|_J\) is surjective.

	\noindent\ref{item1-thm-multHomLiedirdecompideals}.
	Let
	\(
	x=x_I+x_J\in C(\mathfrak{g}),\
	x_I\in I, x_J\in J.
	\)
	For arbitrary \(y_1,\ldots,y_{n-1}\in I\), we have
	\(
	0=[x,y_1,\ldots,y_{n-1}].
	\)
	By multilinearity,
	\[
	0
	=
	[x_I,y_1,\ldots,y_{n-1}]
	+
	[x_J,y_1,\ldots,y_{n-1}].
	\]
	The second bracket is mixed and therefore vanishes. Thus,
	\[
	[x_I,y_1,\ldots,y_{n-1}]=0.
	\]
	Since \(I\) is an \(n\)-Hom-Lie superalgebra, total
	super-skew-symmetry implies that
	\(
	x_I\in C(I).
	\)
	Similarly,
	\(
	x_J\in C(J).
	\)
	Therefore,
	\(
	C(\mathfrak{g})
	\subseteq
	C(I)\oplus C(J).
	\)

	Conversely, let
	\(
	x_I\in C(I),
	x_J\in C(J).
	\)
	Consider a bracket containing \(x_I+x_J\) and arbitrary elements
	of
	\(
	\mathfrak{g}=I\oplus J.
	\)
	By multilinearity, it is the sum of pure and mixed brackets.
    Every pure \(I\)-bracket containing \(x_I\) vanishes because \(x_I\in C(I)\), and every pure \(J\)-bracket containing \(x_J\) vanishes because \(x_J\in C(J)\). All mixed brackets vanish by the first observation. Hence,
	\[
	x_I+x_J\in C(\mathfrak{g}).
	\]
	Consequently,
	\(
	C(\mathfrak{g})=C(I)\oplus C(J).
	\)

	\noindent\ref{item2-thm-multHomLiedirdecompideals}.
	Assume that
	\(
	C(\mathfrak{g})=\{0\}.
	\)
	By part~\ref{item1-thm-multHomLiedirdecompideals},
	\(
	C(I)=C(J)=\{0\}.
	\)
Let \(k\geq0\), and let
	\(
	D\in \operatorname{Der}_{\alpha^{k+1}}(\mathfrak{g})
	\)
	be homogeneous. Let
	\(
	\pi_I:\mathfrak{g}\longrightarrow I,
	\pi_J:\mathfrak{g}\longrightarrow J
	\)
	be the canonical projections.
Let \(x\in I\) and \(y_2,\ldots,y_n\in J\) be homogeneous. Since
	the bracket is mixed,
	\(
	[x,y_2,\ldots,y_n]=0.
	\)
	Applying \(D\) and using the
	\(\alpha^{k+1}\)-derivation identity, we obtain
	\[
	\begin{aligned}
	0
	={}&
	[D(x),\alpha^{k+1}(y_2),\ldots,\alpha^{k+1}(y_n)]
	\\
	&+
	\sum_{r=2}^{n}
	(-1)^{|D|(|x|+|y_2|+\cdots+|y_{r-1}|)}
	\\
	&\qquad\cdot
	[\alpha^{k+1}(x),
	 \alpha^{k+1}(y_2),\ldots,
	 D(y_r),\ldots,
	 \alpha^{k+1}(y_n)].
	\end{aligned}
	\]
	Every term in the sum belongs to \(I\), because it contains
	\(\alpha^{k+1}(x)\in I\) and \(I\) is a Hom-ideal. Therefore,
	its projection onto \(J\) is zero.

Write
	\(
	D(x)=\pi_ID(x)+\pi_JD(x).
	\)
	Since
\[\pi_ID(x)\in I,\quad \alpha^{k+1}(y_2),\ldots,\alpha^{k+1}(y_n)\in J,\]
we have
\(
[\pi_ID(x),\alpha^{k+1}(y_2),\ldots,\alpha^{k+1}(y_n)]
\in I\cap J=\{0\},
\)
because both \(I\) and \(J\) are Hom-ideals. Projecting the preceding
	identity onto \(J\), we obtain
	\[
	[\pi_JD(x),
	 \alpha^{k+1}(y_2),\ldots,\alpha^{k+1}(y_n)]
	=0.
	\]
	Since \(\left.\alpha\right|_J\) is surjective,
	\(\left.\alpha^{k+1}\right|_J\) is also surjective. Hence,
	\[
	[\pi_JD(x),z_2,\ldots,z_n]=0
	\]
	for all \(z_2,\ldots,z_n\in J\). By total
	super-skew-symmetry,
	\[
	\pi_JD(x)\in C(J)=\{0\}.
	\]
	Thus,
	\(
	D(I)\subseteq I.
	\)
	Similarly,
	\(
	D(J)\subseteq J.
	\)
It follows that
	\(
	D_I=\left.D\right|_I
	\in \operatorname{Der}_{\alpha^{k+1}}(I)
	\)
	and
	\(
	D_J=\left.D\right|_J
	\in \operatorname{Der}_{\alpha^{k+1}}(J).
	\)
	Therefore,
	\[
	\operatorname{Der}_{\alpha^{k+1}}(\mathfrak{g})
	\subseteq
	\operatorname{Der}_{\alpha^{k+1}}(I)
	\oplus
	\operatorname{Der}_{\alpha^{k+1}}(J).
	\]

	For the converse inclusion, it is sufficient to work parity-wise. Let
	\[
	D_I\in \operatorname{Der}_{\alpha^{k+1}}(I),
	\qquad
	D_J\in \operatorname{Der}_{\alpha^{k+1}}(J)
	\]
	be homogeneous derivations of the same parity. Define
	\(
	D_I\oplus D_J:\mathfrak{g}\longrightarrow\mathfrak{g}
	\)
	by
	\[
	(D_I\oplus D_J)(i+j)
	=
	D_I(i)+D_J(j),
	\qquad i\in I,\quad j\in J.
	\]
	This map commutes with \(\alpha\). Its derivation identity holds
	on pure \(I\)-brackets and pure \(J\)-brackets because \(D_I\) and \(D_J\) are derivations. On a mixed bracket, both sides of the derivation identity vanish because every term remains mixed.
	Therefore,
	\(
	D_I\oplus D_J
	\in \operatorname{Der}_{\alpha^{k+1}}(\mathfrak{g}).
	\)
	Consequently,
	\[
	\operatorname{Der}_{\alpha^{k+1}}(\mathfrak{g})
	=
	\operatorname{Der}_{\alpha^{k+1}}(I)
	\oplus
	\operatorname{Der}_{\alpha^{k+1}}(J).
	\]

	We now consider the inner derivations. Let
	\[
	X=(x_1,\ldots,x_{n-1})
	\in H(\mathfrak{g})^{n-1}
	\]
	satisfy
	\(
	\alpha(X)=X.
	\)
	For every \(1\leq r\leq n-1\), write
	\[
	x_r=i_r+j_r,
	\qquad
	i_r\in I,\quad j_r\in J.
	\]
	Since \(I\) and \(J\) are graded subspaces, each nonzero \(i_r\) and \(j_r\) has the same parity as \(x_r\). Moreover,
	\(
	\alpha(x_r)=x_r
	\)
	and the uniqueness of the decomposition
	\(\mathfrak{g}=I\oplus J\) imply
	\(
	\alpha(i_r)=i_r,\
	\alpha(j_r)=j_r.
	\)

	By multilinearity, for every \(y\in\mathfrak{g}\),
	\[
	\operatorname{ad}_X^k(y)
	=
	[x_1,\ldots,x_{n-1},\alpha^k(y)]
	\]
	is the sum of brackets obtained by choosing either \(i_r\) or \(j_r\) in every position. Every term involving elements from both \(I\) and \(J\) is mixed and therefore vanishes. Hence,
	\[
	\operatorname{ad}_X^k
	=
	\operatorname{ad}_{(i_1,\ldots,i_{n-1})}^k
	+
	\operatorname{ad}_{(j_1,\ldots,j_{n-1})}^k.
	\]
	The first summand is the zero extension of an inner derivation of
	\(I\), and the second is the zero extension of an inner derivation
	of \(J\). Thus,
	\[
	\operatorname{Inn}_{\alpha^k}(\mathfrak{g})
	\subseteq
	\operatorname{Inn}_{\alpha^k}(I)
	\oplus
	\operatorname{Inn}_{\alpha^k}(J).
	\]

	Conversely, an inner derivation defined by an \(\alpha\)-fixed
	fundamental object in \(I\) is also an inner derivation of
	\(\mathfrak{g}\) and vanishes on \(J\). The same assertion holds with \(I\) and \(J\) interchanged. Therefore,
	\[
	\operatorname{Inn}_{\alpha^k}(\mathfrak{g})
	=
	\operatorname{Inn}_{\alpha^k}(I)
	\oplus
	\operatorname{Inn}_{\alpha^k}(J).
	\]

	\noindent\ref{item3-thm-multHomLiedirdecompideals}.
	First, suppose \(I\) and \(J\) are complete. Then
	\(
	C(I)=C(J)=\{0\}.
	\)
	By part~\ref{item1-thm-multHomLiedirdecompideals},
	\(
	C(\mathfrak{g})=\{0\}.
	\)
	Furthermore, for every \(k\geq0\),
	\(
	\operatorname{Der}_{\alpha^{k+1}}(I)
	=
	\operatorname{Inn}_{\alpha^k}(I)
	\)
	and
	\(
	\operatorname{Der}_{\alpha^{k+1}}(J)
	=
	\operatorname{Inn}_{\alpha^k}(J).
	\)
	Using part~\ref{item2-thm-multHomLiedirdecompideals}, we obtain
	\(
	\operatorname{Der}_{\alpha^{k+1}}(\mathfrak{g})
	=
	\operatorname{Inn}_{\alpha^k}(\mathfrak{g}).
	\)
	Therefore, \(\mathfrak{g}\) is complete.

	Conversely, suppose that \(\mathfrak{g}\) is complete. Then
	\(
	C(\mathfrak{g})=\{0\}.
	\)
	By part~\ref{item1-thm-multHomLiedirdecompideals},
	\(
	C(I)=C(J)=\{0\}.
	\)
	Moreover, for every \(k\geq0\),
	\[
	\begin{aligned}
	\operatorname{Der}_{\alpha^{k+1}}(I)
	\oplus
	\operatorname{Der}_{\alpha^{k+1}}(J)
	&=
	\operatorname{Der}_{\alpha^{k+1}}(\mathfrak{g})
	=
	\operatorname{Inn}_{\alpha^k}(\mathfrak{g})
	\\
	&=
	\operatorname{Inn}_{\alpha^k}(I)
	\oplus
	\operatorname{Inn}_{\alpha^k}(J).
	\end{aligned}
	\]
	Because these are direct sums under the canonical zero extensions, the \(I\)- and \(J\)-components are equal. Hence,
	\(
	\operatorname{Der}_{\alpha^{k+1}}(I)
	=
	\operatorname{Inn}_{\alpha^k}(I)
	\)
	and
	\(
	\operatorname{Der}_{\alpha^{k+1}}(J)
	=
	\operatorname{Inn}_{\alpha^k}(J).
	\)
	Therefore, both \(I\) and \(J\) are complete.
\end{proof}

The above theorem shows that completeness is preserved under direct sums of Hom-ideals with a bijective twisting map. Conversely, the
completeness of such a direct sum implies the completeness of each
summand.



We now consider the recursive construction of \(n\)-ary
Hom-Nambu superalgebras from multiplicative Hom-Lie
superalgebras. The resulting bracket is generally not totally
super-skew-symmetric and therefore does not, in general, define an
\(n\)-Hom-Lie superalgebra.

\begin{theorem}[\cite{MabroukNcibSilvAACA2021-GenDerRotaBaxterOpsnaryHomNambuSuperal}]
Let \( (\mathfrak g,[\cdot,\cdot],\alpha) \) be a multiplicative Hom-Lie superalgebra. For \(n\geq2\), define
\(
[x_1,x_2]_2=[x_1,x_2]
\)
and
\[
[x_1,\ldots,x_n]_n
=
\bigl[
[x_1,\ldots,x_{n-1}]_{n-1},
\alpha^{n-2}(x_n)
\bigr].
\]
Then
\(
\mathfrak g_n
=
\bigl(
\mathfrak g,
[\cdot,\ldots,\cdot]_n,
\beta_n
\bigr),
\
\beta_n=\alpha^{n-1},
\)
is a multiplicative \(n\)-ary Hom-Nambu superalgebra.
\end{theorem}

This theorem shows that every multiplicative Hom-Lie superalgebra
naturally generates a sequence of multiplicative \(n\)-ary Hom-Nambu
superalgebras, whose twisting maps are
\(
\beta_n=\alpha^{n-1}.
\)
\begin{lemma}\label{lem:center-induced}
	Let
	\(
	(\mathfrak{g},[\cdot,\cdot],\alpha)
	\)
	be a multiplicative Hom-Lie superalgebra with a surjective
	twisting map \(\alpha\). For every \(n\geq2\), let
	\[
	\mathfrak{g}_n
	=
	\bigl(
	\mathfrak{g},
	[\cdot,\ldots,\cdot]_n,
	\beta_n
	\bigr),
	\qquad
	\beta_n=\alpha^{n-1},
	\]
	be the recursively induced multiplicative \(n\)-ary
	Hom-Nambu superalgebra with
	\begin{align*}
	& [x_1,x_2]_2=[x_1,x_2],
	\\
	& [x_1,\ldots,x_n]_n
	=
	\bigl[
	[x_1,\ldots,x_{n-1}]_{n-1},
	\alpha^{n-2}(x_n)
	\bigr] \ \text{for}\ n\geq 3.
	\end{align*}

	For every \(1\leq i\leq n\), define the \(i\)-th positional
	center of \(\mathfrak{g}_n\) by
	\[
	C_i(\mathfrak{g}_n)
	=
	\left\{
	x\in\mathfrak{g}
	\ \middle|\
	[y_1,\ldots,y_{i-1},x,
	y_{i+1},\ldots,y_n]_n=0
	\text{ for all }y_j\in\mathfrak{g}
	\right\},
	\]
	and set
	\(
	C(\mathfrak{g}_n)
	=
	\bigcap\limits_{i=1}^{n}C_i(\mathfrak{g}_n).
	\)
	Then
	\(
	C(\mathfrak{g})=\{0\}
	\Leftrightarrow
	C(\mathfrak{g}_n)=\{0\}.
	\)
\end{lemma}

\begin{proof}
	Assume first that
	\(
	C(\mathfrak{g})=\{0\}.
	\)
	We prove by induction on \(m\geq2\) that
	\(
	C_1(\mathfrak{g}_m)=\{0\}.
	\)

	For \(m=2\), we have
	\( \mathfrak{g}_2=\mathfrak{g}, \)
	and therefore
	\(C_1(\mathfrak{g}_2)=C(\mathfrak{g})=\{0\}.\)

	Suppose that
	\(
	C_1(\mathfrak{g}_m)=\{0\}
	\)
	for some \(m\geq2\), and let
	\(
	x\in C_1(\mathfrak{g}_{m+1}).
	\)
	For arbitrary elements
	\(
	y_1,\ldots,y_{m-1},z\in\mathfrak{g},
	\)
	we have
	\(
	[x,y_1,\ldots,y_{m-1},z]_{m+1}=0.
	\)
	Using the recursive definition of the induced bracket, we obtain
	\[
	\bigl[
	[x,y_1,\ldots,y_{m-1}]_m,
	\alpha^{m-1}(z)
	\bigr]
	=0.
	\]
	Since \(\alpha\) is surjective, \(\alpha^{m-1}\) is also
	surjective. Hence, for every \(w\in\mathfrak{g}\), there exists
	\(z\in\mathfrak{g}\) such that
	\(
	w=\alpha^{m-1}(z).
	\)
	Consequently,
	\[
	\bigl[
	[x,y_1,\ldots,y_{m-1}]_m,w
	\bigr]
	=0
	\]
	for every \(w\in\mathfrak{g}\). Thus,
	\(
	[x,y_1,\ldots,y_{m-1}]_m
	\in C(\mathfrak{g}).
	\)
	Since \(C(\mathfrak{g})=\{0\}\), it follows that
	\(
	[x,y_1,\ldots,y_{m-1}]_m=0.
	\)
	As \(y_1,\ldots,y_{m-1}\) were arbitrary, we obtain
	\(
	x\in C_1(\mathfrak{g}_m).
	\)
	The induction hypothesis now yields \(x=0\). Thus,
	\(
	C_1(\mathfrak{g}_{m+1})=\{0\}.
	\)
It follows by induction that
	\(
	C_1(\mathfrak{g}_n)=\{0\}.
	\)
	Since
	\[
	C(\mathfrak{g}_n)
	=
	\bigcap_{i=1}^{n}C_i(\mathfrak{g}_n)
	\subseteq C_1(\mathfrak{g}_n),
	\]
	we conclude that
	\(
	C(\mathfrak{g}_n)=\{0\}.
	\)

	Conversely, assume that
	\(
	C(\mathfrak{g}_n)=\{0\}.
	\)
	Let \(x\in C(\mathfrak{g})\). We first observe that
	\(
	\alpha^r(x)\in C(\mathfrak{g})
	\)
	for every integer \(r\geq0\).

	Indeed, let \(y\in\mathfrak{g}\). Since \(\alpha\) is
	surjective, \(\alpha^r\) is surjective, and hence there exists
	\(z\in\mathfrak{g}\) such that
	\(
	y=\alpha^r(z).
	\)
	By multiplicativity,
	\[
	[\alpha^r(x),y]
	=
	[\alpha^r(x),\alpha^r(z)]
	=
	\alpha^r([x,z])
	=0.
	\]
	Therefore,
	\(
	\alpha^r(x)\in C(\mathfrak{g}).
	\)

	We now show that \(x\) belongs to every positional center of
	\(\mathfrak{g}_n\). More precisely, we prove by induction on
	\(m\geq2\) that, for every \(1\leq i\leq m\),
	\[
	[y_1,\ldots,y_{i-1},x,
	y_{i+1},\ldots,y_m]_m=0
	\]
	for all \(y_j\in\mathfrak{g}\).

	For \(m=2\), this follows from
	\(
	x\in C(\mathfrak{g}).
	\)

	Suppose that the assertion holds for \(m\), and consider the
	\((m+1)\)-ary bracket.

	First, let \(1\leq i\leq m\). By the recursive definition,
	\[
	\begin{aligned}
	&[y_1,\ldots,y_{i-1},x,
	y_{i+1},\ldots,y_m,y_{m+1}]_{m+1}
	\\
	&\quad=
	\bigl[
	[y_1,\ldots,y_{i-1},x,
	y_{i+1},\ldots,y_m]_m,
	\alpha^{m-1}(y_{m+1})
	\bigr].
	\end{aligned}
	\]
	The inner \(m\)-ary bracket vanishes by the induction
	hypothesis. Hence,
	\[
	[y_1,\ldots,y_{i-1},x,
	y_{i+1},\ldots,y_m,y_{m+1}]_{m+1}=0.
	\]

	Now let \(i=m+1\). Then
	\[
	\begin{aligned} \relax
	[y_1,\ldots,y_m,x]_{m+1}
	&=
	\bigl[
	[y_1,\ldots,y_m]_m,
	\alpha^{m-1}(x)
	\bigr].
	\end{aligned}
	\]
	As proved above,
	\(
	\alpha^{m-1}(x)\in C(\mathfrak{g}).
	\)
	Therefore,
	\[
	\bigl[
	[y_1,\ldots,y_m]_m,
	\alpha^{m-1}(x)
	\bigr]
	=0.
	\]
	Thus the assertion also holds when \(x\) occupies the last
	position.

	By induction, \(
	x\in C_i(\mathfrak{g}_n).
	\) for every \(1\leq i\leq n\).
	Hence,
	\[
	x\in\bigcap_{i=1}^{n}C_i(\mathfrak{g}_n)
	=C(\mathfrak{g}_n).
	\]
	Since \(C(\mathfrak{g}_n)=\{0\}\), it follows that \(x=0\).
	Therefore,
	\(
	C(\mathfrak{g})=\{0\}.
	\)

	This completes the proof.
\end{proof}

The preceding lemma shows that, when the twisting map is
surjective, the recursive construction preserves and reflects the
triviality of the center. Here, the center of the induced
\(n\)-ary Hom-Nambu superalgebra is understood as the intersection of all its positional centers.
\begin{lemma}\label{lem:induced-hom-ideal}
	Let
	\(
	(\mathfrak{g},[\cdot,\cdot],\alpha)
	\)
	be a multiplicative Hom-Lie superalgebra, and
	\(I\) be a Hom-ideal of \(\mathfrak{g}\). For every \(n\geq2\),
	consider the recursively induced multiplicative \(n\)-ary
	Hom-Nambu superalgebra
	\(
	\mathfrak{g}_n
	=
	\bigl(
	\mathfrak{g},
	[\cdot,\ldots,\cdot]_n,
	\beta_n
	\bigr),
	\
	\beta_n=\alpha^{n-1},
	\)
	where
	\begin{align*}
	 [x_1,x_2]_2 &= [x_1,x_2],
	\\
	 [x_1,\ldots,x_n]_n
    &=
	\bigl[
	[x_1,\ldots,x_{n-1}]_{n-1},
	\alpha^{n-2}(x_n)
	\bigr] \ \text{for}\ n\geq 3.
	\end{align*}
Then, \(I\) is a Hom-ideal of \(\mathfrak{g}_n\). More precisely,
for each \(1\leq i\leq n\),
	\[
	[\mathfrak{g},\ldots,\mathfrak{g},
	\underset{i\text{-th position}}{I},
	\mathfrak{g},\ldots,\mathfrak{g}]_n
	\subseteq I \ \text{and}\ \beta_n(I)\subseteq I.
	\]
\end{lemma}

\begin{proof}
	Since \(I\) is a Hom-ideal of the binary Hom-Lie
	superalgebra \(\mathfrak{g}\), it is a graded subspace satisfying
	\(
	\alpha(I)\subseteq I
	\)
	and
	\(
	[I,\mathfrak{g}]\subseteq I.
	\)
	By the super-skew-symmetry of the binary bracket, we also have
	\(
	[\mathfrak{g},I]\subseteq I.
	\)
	Indeed, for homogeneous \(x\in\mathfrak{g}\) and \(y\in I\),
	\[
	[x,y]
	=
	-(-1)^{|x||y|}[y,x]\in I.
	\]
	By linearity, the same inclusion holds for arbitrary
	\(x\in\mathfrak{g}\) and \(y\in I\).

	We first verify the stability under the twisting map of \(\mathfrak{g}_n\). Since
	\(
	\alpha(I)\subseteq I,
	\)
	the repeated application of \(\alpha\) gives
	\(
	\alpha^r(I)\subseteq I
	\)
	for every integer \(r\geq1\). In particular,
	\(
	\beta_n(I)
	=
	\alpha^{n-1}(I)
	\subseteq I.
	\)

	It remains to be proven that \(I\) is stable under the induced
	\(n\)-ary bracket in every argument. We proceed by induction on
	\(n\).

	For \(n=2\), the induced bracket is the original binary bracket
	\[
	[x_1,x_2]_2=[x_1,x_2].
	\]
	The required inclusions
	\(
	[I,\mathfrak{g}]_2\subseteq I\)
	and
	\([\mathfrak{g},I]_2\subseteq I
	\)
	follow from the Hom-ideal property of \(I\) and the super-skew-symmetry of the binary bracket.

	Suppose that, for some \(n\geq2\), the assertion holds for the
	induced \(n\)-ary bracket. Thus, for every \(1\leq i\leq n\),
	\[
	[\mathfrak{g},\ldots,\mathfrak{g},
	\underset{i\text{-th position}}{I},
	\mathfrak{g},\ldots,\mathfrak{g}]_n
	\subseteq I.
	\]
We prove the corresponding assertion for the
	\((n+1)\)-ary bracket.

Let
	\(
	x_1,\ldots,x_{n+1}\in\mathfrak{g}
	\)
	be homogeneous.
First, suppose that
	\(
	x_i\in I
	\)
	for some \(1\leq i\leq n\). By the induction hypothesis,
	\(
	[x_1,\ldots,x_n]_n\in I.
	\)
	Using the recursive definition of the induced bracket, we obtain
	\[
	\begin{aligned} \relax
	[x_1,\ldots,x_n,x_{n+1}]_{n+1}
	&=
	\bigl[
	[x_1,\ldots,x_n]_n,
	\alpha^{n-1}(x_{n+1})
	\bigr].
	\end{aligned}
	\]
	Since
	\(
	[x_1,\ldots,x_n]_n\in I
	\)
	and
	\(
	\alpha^{n-1}(x_{n+1})\in\mathfrak{g},
	\)
	the inclusion
	\(
	[I,\mathfrak{g}]\subseteq I
	\)
	implies that
	\(
	[x_1,\ldots,x_n,x_{n+1}]_{n+1}\in I.
	\)

	Now suppose that the element belonging to \(I\) occurs in the
	last position, that is,
	\(
	x_{n+1}\in I.
	\)
	Since \(I\) is stable under \(\alpha\), we have
	\(
	\alpha^{n-1}(x_{n+1})\in I.
	\)
	Moreover,
	\(
	[x_1,\ldots,x_n]_n\in\mathfrak{g}.
	\)
	Therefore,
	\[
	[x_1,\ldots,x_n,x_{n+1}]_{n+1}
	=
	\bigl[
	[x_1,\ldots,x_n]_n,
	\alpha^{n-1}(x_{n+1})
	\bigr]
	\in [\mathfrak{g},I]
	\subseteq I.
	\]

	Hence, for every \(1\leq i\leq n+1\),
	\[
	[\mathfrak{g},\ldots,\mathfrak{g},
	\underset{i\text{-th position}}{I},
	\mathfrak{g},\ldots,\mathfrak{g}]_{n+1}
	\subseteq I.
	\]
	By induction, this inclusion holds for all \(n\geq2\).

	Finally, taking all arguments in \(I\) gives
	\(
	[I,\ldots,I]_n\subseteq I.
	\)
	Thus, \(I\) is also a Hom-subalgebra of \(\mathfrak{g}_n\).
	Together with
	\(
	\beta_n(I)\subseteq I
	\)
	and stability in every argument, this proves that \(I\) is a
	Hom-ideal of \(\mathfrak{g}_n\).
\end{proof}

Thus, every Hom-ideal of the original binary Hom-Lie
superalgebra remains a Hom-ideal of each recursively induced
\(n\)-ary Hom-Nambu superalgebra. In particular, the recursive
construction preserves the ideal structure of the underlying
Hom-Lie superalgebra.

\begin{proposition}\label{prop:relative-der-induced}
	Let
	\(
	(\mathfrak{g},[\cdot,\cdot],\alpha)
	\)
	be a multiplicative Hom-Lie superalgebra, and let
	\(
	D:\mathfrak{g}\longrightarrow\mathfrak{g}
	\)
	be an \(\alpha^k\)-derivation of \(\mathfrak{g}\), where
	\(k\geq0\). Then \(D\) commutes with \(\beta_n=\alpha^{n-1}\)
	and satisfies
	\[
	\begin{aligned}
	D([x_1,\ldots,x_n]_n)
	={}&
	\sum_{i=1}^{n}
	(-1)^{|D|(|x_1|+\cdots+|x_{i-1}|)}
	\\
	&\quad\cdot
	[\alpha^k(x_1),\ldots,\alpha^k(x_{i-1}),
	D(x_i),
	\\
	&\hspace{35mm}
	\alpha^k(x_{i+1}),\ldots,\alpha^k(x_n)]_n
	\end{aligned}
	\]
	for all homogeneous
	\(x_1,\ldots,x_n\in\mathfrak{g}\).
\end{proposition}
\begin{proof}
Since \(D\) commutes with \(\alpha\), it also commutes with every
power of \(\alpha\). In particular,
\(
D\circ\beta_n
=
D\circ\alpha^{n-1}
=
\alpha^{n-1}\circ D
=
\beta_n\circ D.
\)

We prove the displayed derivation identity using the induction in \(n\).

For \(n=2\), the assertion is precisely the
\(\alpha^k\)-derivation identity for the original Hom-Lie
superalgebra \(\mathfrak g\).

Suppose that the assertion holds for some \(n\geq2\). We first note
that, by multiplicativity of the original Hom-Lie superalgebra,
\[
\alpha^r([x_1,\ldots,x_n]_n)
=
[\alpha^r(x_1),\ldots,\alpha^r(x_n)]_n
\]
for every \(r\geq0\). This follows by induction from the recursive
definition of the brackets.

Let \(x_1,\ldots,x_{n+1}\in H(\mathfrak g)\). By the recursive
definition,
\[
[x_1,\ldots,x_{n+1}]_{n+1}
=
\bigl[
[x_1,\ldots,x_n]_n,
\alpha^{n-1}(x_{n+1})
\bigr].
\]
Since \(D\) is an \(\alpha^k\)-derivation of the binary Hom-Lie
superalgebra, we have
\[
\begin{aligned}
&D([x_1,\ldots,x_{n+1}]_{n+1})\\
&=
\bigl[
D([x_1,\ldots,x_n]_n),
\alpha^{k+n-1}(x_{n+1})
\bigr]
\\
&\quad+
(-1)^{|D|(|x_1|+\cdots+|x_n|)}
\bigl[
\alpha^k([x_1,\ldots,x_n]_n),
D(\alpha^{n-1}(x_{n+1}))
\bigr].
\end{aligned}
\]
Here we used
\(
\bigl|[x_1,\ldots,x_n]_n\bigr|
=
|x_1|+\cdots+|x_n|.
\)

By the induction hypothesis,
\[
\begin{aligned}
D([x_1,\ldots,x_n]_n)
={}&
\sum_{i=1}^{n}
(-1)^{|D|(|x_1|+\cdots+|x_{i-1}|)}
\\
&\quad\cdot
[\alpha^k(x_1),\ldots,\alpha^k(x_{i-1}),
D(x_i),
\alpha^k(x_{i+1}),\ldots,\alpha^k(x_n)]_n .
\end{aligned}
\]
Substituting this expression into the first term and using the
recursive definition of the \((n+1)\)-ary bracket gives
\[
\begin{aligned}
&\bigl[
D([x_1,\ldots,x_n]_n),
\alpha^{k+n-1}(x_{n+1})
\bigr]
\\
&=
\sum_{i=1}^{n}
(-1)^{|D|(|x_1|+\cdots+|x_{i-1}|)}
\\
&\qquad\cdot
[\alpha^k(x_1),\ldots,\alpha^k(x_{i-1}),
D(x_i),
\alpha^k(x_{i+1}),\ldots,
\alpha^k(x_{n+1})]_{n+1}.
\end{aligned}
\]

For the second term, since \(D\) commutes with \(\alpha\), we have
\[
D(\alpha^{n-1}(x_{n+1}))
=
\alpha^{n-1}(D(x_{n+1})),
\]
while
\(
\alpha^k([x_1,\ldots,x_n]_n)
=
[\alpha^k(x_1),\ldots,\alpha^k(x_n)]_n.
\)
Therefore,
\[
\begin{aligned}
&\bigl[
\alpha^k([x_1,\ldots,x_n]_n),
D(\alpha^{n-1}(x_{n+1}))
\bigr]
\\
&=
\bigl[
[\alpha^k(x_1),\ldots,\alpha^k(x_n)]_n,
\alpha^{n-1}(D(x_{n+1}))
\bigr]
\\
&=
[\alpha^k(x_1),\ldots,\alpha^k(x_n),D(x_{n+1})]_{n+1}.
\end{aligned}
\]

Combining the two expressions, we obtain
\begin{multline*}
D([x_1,\ldots,x_{n+1}]_{n+1})
=
\sum_{i=1}^{n+1}
(-1)^{|D|(|x_1|+\cdots+|x_{i-1}|)}
\\
\cdot
[\alpha^k(x_1),\ldots,\alpha^k(x_{i-1}),
D(x_i),
\alpha^k(x_{i+1}),\ldots,
\alpha^k(x_{n+1})]_{n+1}.
\end{multline*}

Thus the assertion holds for \(n+1\), and hence for every
\(n\geq2\).
\end{proof}

\begin{remark}
	The preceding proposition gives only a relative derivation identity.
	A \(\beta_n^r\)-derivation of the \(n\)-ary Hom-Nambu superalgebra
\(\mathfrak{g}_n\), in the sense of the definition of an
\(\alpha^k\)-derivation, satisfies
	the derivation identity with
	\(
	\beta_n^r
	=
	\alpha^{(n-1)r}
	\)
	applied to all unchanged arguments.
Consequently, if \(D\) is an
	\(\alpha^{(n-1)r}\)-derivation of the original binary Hom-Lie
	superalgebra, then Proposition~\ref{prop:relative-der-induced}
	shows that \(D\) is a \(\beta_n^r\)-derivation of
	\(\mathfrak{g}_n\). The converse is not asserted.
In particular, this result does not imply that the completeness of
	\(\mathfrak{g}\) is inherited by \(\mathfrak{g}_n\).
\end{remark}

\section{Complete Low-Dimensional \texorpdfstring{$n$}{n}-Hom-Lie Superalgebras} \label{sec4-ComplLowdimnHomLieSuperalg}
In this section, we determine the complete low-dimensional Hom-Lie
superalgebras by computing their centers, derivations, and inner
derivations. Recall that a multiplicative Hom-Lie superalgebra is
complete if its center is trivial and
\(
\operatorname{Der}_{\alpha^{k+1}}(\mathfrak{g})
=
\operatorname{Inn}_{\alpha^k}(\mathfrak{g})
\)
for every \(k\geq0\).
	
Next, we determine which of the $2$-dimensional Hom-Lie superalgebras classified in Theorem~\ref{2dimen} are complete in the sense of Definition~\ref{def:complete}. Throughout this subsection, all Hom-Lie superalgebras are defined
over the field of complex numbers $\mathbb{C}$.
	
\begin{theorem}\label{complete2dim}
	Among the \(2\)-dimensional multiplicative Hom-Lie superalgebras
	listed in Theorem~\ref{2dimen}, the complete ones are precisely
	\(
	\mathfrak{g}^{3,b}_{1;1},\
	b\in\mathbb{C}^{*}.
	\)
\end{theorem}

\begin{proof}
	The abelian Hom-Lie superalgebras
	\(
	\mathfrak{g}^{1}_{0;2}\) and \(\mathfrak{g}^{2}_{1;1}
	\)
	are not complete because their centers are nonzero.

	For
	\(
	\mathfrak{g}^{5,b}_{1;1},
	\)
	we have
	\[
	[e_0,e_1]=0
	\qquad\text{and}\qquad
	[e_1,e_1]=e_0.
	\]
	Therefore,
	\(
	e_0\in C(\mathfrak{g}^{5,b}_{1;1}),
	\)
	and hence, \(\mathfrak{g}^{5,b}_{1;1}\) is not complete.

	We now consider
	\(
	\mathfrak{g}
	=
	\mathfrak{g}^{3,b}_{1;1},
	\)
	where
	\[
	[e_0,e_1]=e_1,
	\qquad
	[e_1,e_1]=0,
	\qquad
	\alpha(e_0)=e_0,
	\qquad
	\alpha(e_1)=be_1.
	\]
	First, we show that
	\(
	C(\mathfrak{g})=\{0\}.
	\)
	Let
	\(
	x=ue_0+ve_1\in C(\mathfrak{g}).
	\)
	Then
	\[
	[x,e_0]=-ve_1=0,
	\quad
	[x,e_1]=ue_1=0.
	\]
	Thus, \(u=v=0\), and consequently
	\(
	C(\mathfrak{g})=\{0\}.
	\)

	Assume first that \(b\neq0\). Let \(k\geq0\), and let
	\(
	D\in \operatorname{Der}_{\alpha^{k+1}}(\mathfrak{g})
	\)
	be homogeneous.
Suppose that \(D\) is even. Then
	\(
	D(e_0)=ue_0,
	\
	D(e_1)=ve_1
	\)
	for some \(u,v\in\mathbb{K}\). Applying \(D\) to
	\(
	[e_0,e_1]=e_1,
	\)
	we obtain
	\[
	ve_1
	=
	D([e_0,e_1]) =
	[D(e_0),\alpha^{k+1}(e_1)]
	+
	[\alpha^{k+1}(e_0),D(e_1)]
	=
	ub^{k+1}e_1+ve_1.
	\]
	Since \(b\neq0\), it follows that
	\(
	u=0.
	\)
	Thus,
	\(
	D(e_0)=0,\
	D(e_1)=ve_1.
	\)
	The element
	\(
	x=vb^{-k}e_0
	\)
	satisfies
	\(
	\alpha(x)=x.
	\)
	Moreover,
	\(
	\operatorname{ad}_x^k(e_0)=0
	\)
	and
	\[
	\operatorname{ad}_x^k(e_1)
	=
	[x,\alpha^k(e_1)]
	=
	[vb^{-k}e_0,b^ke_1]
	=
	ve_1.
	\]
	Therefore,
	\(
	D=\operatorname{ad}_x^k
	\in
	\operatorname{Inn}_{\alpha^k}(\mathfrak{g}).
	\)

	Suppose that \(D\) is odd. Then
	\(
	D(e_0)=ue_1,
	\
	D(e_1)=ve_0
	\)
	for some \(u,v\in\mathbb{K}\). The equality
	\(
	D\circ\alpha=\alpha\circ D
	\)
	gives
	\[
	(1-b)u=0
	\qquad\text{and}\qquad
	(b-1)v=0.
	\]
	Applying \(D\) to
	\(
	[e_0,e_1]=e_1
	\)
	gives
	\[
	ve_0
	=
	[D(e_0),\alpha^{k+1}(e_1)]
	+
	[\alpha^{k+1}(e_0),D(e_1)]
	=
	0.
	\]
	Hence,
	\(
	v=0.
	\)

	If \(b\neq1\), then \(u=0\) and consequently \(D=0\). If \(b=1\), then
	\(
	D(e_0)=ue_1,
	\
	D(e_1)=0.
	\)
	Since \(\alpha(e_1)=e_1\), the homogeneous element
	\(
	x=-ue_1
	\)
	is fixed by \(\alpha\), and
	\[
	\operatorname{ad}_x^k(e_0)
	=
	[-ue_1,e_0]
	=
	ue_1,
	\]
	while
	\(
	\operatorname{ad}_x^k(e_1)=0.
	\)
	Therefore,
	\(
	D=\operatorname{ad}_x^k
	\in
	\operatorname{Inn}_{\alpha^k}(\mathfrak{g}).
	\)

	Thus, for every \(b\neq0\) and every \(k\geq0\),
	\[
	\operatorname{Der}_{\alpha^{k+1}}(\mathfrak{g}^{3,b}_{1;1})
	=
	\operatorname{Inn}_{\alpha^k}
	(\mathfrak{g}^{3,b}_{1;1}).
	\]
	Since its center is trivial,
	\(
	\mathfrak{g}^{3,b}_{1;1}
	\)
	is complete for all \(b\in\mathbb{C}^{*}\).

	Let \(b=0\). Define the even linear map \(D\) by
	\(
	D(e_0)=e_0,
	D(e_1)=0.
	\)
	It is straightforward to verify that
	\(
	D\in \operatorname{Der}_{\alpha^{k+1}}
	(\mathfrak{g}^{3,0}_{1;1})
	\)
	for every \(k\geq0\). However, every inner derivation of
	\(\mathfrak{g}^{3,0}_{1;1}\) vanishes on \(e_0\), whereas
	\(
	D(e_0)=e_0.
	\)
	Hence, \(D\) is not inner, and
	\(
	\mathfrak{g}^{3,0}_{1;1}
	\)
	is not complete.

	Finally, consider
	\(
	\mathfrak{g}
	=
	\mathfrak{g}^{4,a}_{1;1},
	\
	a\neq0,1.
	\)
	Define the even linear map \(D\) by
	\[
	D(e_0)=e_0,
	\qquad
	D(e_1)=0.
	\]
	Then
	\(
	D\circ\alpha=\alpha\circ D,
	\)
	and a direct verification shows that
	\(
	D\in \operatorname{Der}_{\alpha^{k+1}}(\mathfrak{g})
	\)
	for every \(k\geq0\).

	On the other hand, there is no nonzero homogeneous element
	\(x\in\mathfrak{g}\) satisfying
	\(
	\alpha(x)=x.
	\)
	Therefore,
	\(
	\operatorname{Inn}_{\alpha^k}(\mathfrak{g})
	=
	\{0\}
	\)
	for every \(k\geq0\). Since \(D\neq0\), it follows that
	\(
	\mathfrak{g}^{4,a}_{1;1}
	\)
	is not complete.
\end{proof}

We now determine the complete $3$-dimensional multiplicative Hom-Lie
superalgebras occurring in the classification of Wang, Zhang, and Wei
\cite{wang2016}. Throughout this subsection, all Hom-Lie superalgebras
are defined over $\mathbb{K}=\mathbb{C}$, and
\(
\mathfrak{g}_{\bar 0}
=
\operatorname{span}\{e_1,e_2\},
\
\mathfrak{g}_{\bar 1}
=
\operatorname{span}\{e_3\}.
\)
We retain the notation $L^i_{2;1}$ used in \cite{wang2016}. The
matrices of the twisting maps are interpreted by columns.

Every linear map
\(
D:\mathfrak{g}\rightarrow\mathfrak{g}
\)
has a unique decomposition
\(
D=D_{\bar 0}+D_{\bar 1},
\)
where \(D_{\bar 0}\) is even and \(D_{\bar 1}\) is odd. Since the
twisting map and the bracket are even, if
\(
D\in \operatorname{Der}_{\alpha^k}(\mathfrak{g}),
\)
then both homogeneous components satisfy
\[
D_{\bar 0}\circ\alpha=\alpha\circ D_{\bar 0},
\qquad
D_{\bar 1}\circ\alpha=\alpha\circ D_{\bar 1},
\]
and each component separately satisfies the
\(\alpha^k\)-derivation identity. Consequently,
\(
\operatorname{Der}_{\alpha^k}(\mathfrak{g})
=
\operatorname{Der}_{\alpha^k}(\mathfrak{g})_{\bar 0}
\oplus
\operatorname{Der}_{\alpha^k}(\mathfrak{g})_{\bar 1}.
\)

It is therefore sufficient to compute the even and odd derivations
separately. Relative to the ordered homogeneous basis
\(\{e_1,e_2,e_3\}\), they have the forms
\[
D_{\bar 0}
=
\begin{pmatrix}
	m&n&0\\
	q&r&0\\
	0&0&v
\end{pmatrix},
\qquad
D_{\bar 1}
=
\begin{pmatrix}
	0&0&p\\
	0&0&s\\
	t&u&0
\end{pmatrix}.
\]

\begin{lemma}\label{lem:center-screening-3dim}
Among the $3$-dimensional complex multiplicative Hom-Lie
superalgebras of superdimension $(2|1)$ listed in \textnormal{\cite{wang2016}},
the families with trivial center are precisely
\[
L^6_{2;1},\quad
L^{9,a}_{2;1},\quad
L^{10}_{2;1},\quad
L^{11,a}_{2;1},\quad
L^{12,a}_{2;1},
\]
\[
L^{16,a}_{2;1},\quad
L^{17}_{2;1},\quad
L^{18,a,b}_{2;1},\quad
L^{19,a,b}_{2;1},\quad
L^{20,a,b,c}_{2;1},
\quad\text{and}\quad
L^{23}_{2;1}.
\]
All the remaining families in the Wang classification have nonzero
center, and hence are not complete.
\end{lemma}

\begin{proof}
The assertion follows by applying the definition of the center to the
multiplication tables in \cite{wang2016}. More explicitly, the centers
of the families excluded above are
\[
\begin{array}{c|c}
\text{families}&C(\mathfrak{g})\\ \hline
L^2_{2;1}&\mathfrak{g}\\
L^3_{2;1}&\operatorname{span}\{e_1,e_2\}\\
L^4_{2;1},\,L^5_{2;1}&\operatorname{span}\{e_3\}\\
L^7_{2;1},\,L^8_{2;1}&\operatorname{span}\{e_1\}\\
L^{13}_{2;1},\,L^{14}_{2;1},\,L^{15}_{2;1}
&\operatorname{span}\{e_1\}\\
L^{21}_{2;1},\,L^{22}_{2;1}
&\operatorname{span}\{e_2\}.
\end{array}
\]
A direct calculation gives zero center for each of the families listed
in the statement.
\end{proof}

\begin{lemma}\label{lem:L6-L9}
The centerless families $L^6_{2;1}$ and $L^{9,a}_{2;1}$ are not
complete.

More precisely, for $L^6_{2;1}$ one has
\[
[e_1,e_2]=e_1,
\qquad
[e_3,e_3]=e_1,
\]
and the correct column interpretation of the twisting matrix in
\textnormal{\cite{wang2016}} is
\[
\alpha(e_1)=0,
\qquad
\alpha(e_2)=e_1,
\qquad
\alpha(e_3)=0.
\]
For every integer $k\geq 1$,
\[
\mathrm{\operatorname{Der}}_{\alpha^k}(\mathfrak{g})_{\bar 0}
=
\left\{
\begin{pmatrix}
0&n&0\\
0&0&0\\
0&0&v
\end{pmatrix}
\;\middle|\;
n,v\in\mathbb{K}
\right\},
\]
and
\[
\mathrm{\operatorname{Der}}_{\alpha^k}(\mathfrak{g})_{\bar 1}
=
\left\{
\begin{pmatrix}
0&0&p\\
0&0&0\\
0&u&0
\end{pmatrix}
\;\middle|\;
p,u\in\mathbb{K}
\right\}.
\]
Moreover,
\(
\mathrm{Inn}_{\alpha^{k-1}}(\mathfrak{g})=\{0\}.
\)

For $L^{9,a}_{2;1}$, where $a\in\mathbb{K}^{*}$,
\begin{align*}
& [e_1,e_2]=e_1,
\qquad
[e_3,e_3]=e_2,
\\
& \alpha(e_1)=0,
\qquad
\alpha(e_2)=a^2e_2,
\qquad
\alpha(e_3)=ae_3.
\end{align*}
For every integer $k\geq 1$,
\[
\mathrm{\operatorname{Der}}_{\alpha^k}(\mathfrak{g})_{\bar 0}
=
\left\{
\begin{pmatrix}
m&0&0\\
0&2a^kv&0\\
0&0&v
\end{pmatrix}
\;\middle|\;
m,v\in\mathbb{K},
\quad
(1-a^{2k})m=0
\right\},
\]
whereas
\[
\mathrm{\operatorname{Der}}_{\alpha^k}(\mathfrak{g})_{\bar 1}
=
\begin{cases}
\displaystyle
\left\{
\begin{pmatrix}
0&0&0\\
0&0&s\\
0&0&0
\end{pmatrix}
\;\middle|\;
s\in\mathbb{K}
\right\},
& a=1,\\[2.2em]
\{0\},
& a\neq1.
\end{cases}
\]
In particular, $L^{9,a}_{2;1}$ is not complete for any
$a\in\mathbb{K}^{*}$.
\end{lemma}

\begin{proof}
Let $D$ be homogeneous and impose
\(
D\circ\alpha=\alpha\circ D
\)
together with the $\alpha^k$-derivation identity on all ordered pairs of homogeneous basis elements.

For $L^6_{2;1}$, these equations give
\begin{align*}
D_{\bar0}(e_1)=0,\qquad
D_{\bar0}(e_2)=ne_1,\qquad
D_{\bar0}(e_3)=ve_3,
\\
D_{\bar1}(e_1)=0,\qquad
D_{\bar1}(e_2)=ue_3,\qquad
D_{\bar1}(e_3)=pe_1.
\end{align*}
Since $\alpha(x)=x$ implies $x=0$, every inner derivation is zero.
The displayed derivation spaces are nonzero, so the algebra is not
complete.

For $L^{9,a}_{2;1}$, an even linear map commuting with $\alpha$ is
diagonal. Applying the derivation identity to
\[
[e_1,e_2]=e_1
\qquad\text{and}\qquad
[e_3,e_3]=e_2
\]
gives
\(
(1-a^{2k})m=0,
\
r=2a^kv.
\)
For an odd linear map, the commutation relation and derivation identity give $D=0$ unless $a=1$; when $a=1$, the only possible odd derivations satisfy $D(e_3)=se_2$.

Finally, the even derivation
\[
\begin{pmatrix}
0&0&0\\
0&2a^k&0\\
0&0&1
\end{pmatrix}
\]
is not inner. Hence, $L^{9,a}_{2;1}$ is never complete.
\end{proof}

\begin{lemma}\label{lem:L10-L12}
Consider the bracket
\[
[e_1,e_2]=e_1,
\qquad
[e_1,e_3]=e_3,
\qquad
[e_2,e_3]=[e_3,e_3]=0.
\]

\begin{enumerate}[label=\textup{\rm(\roman*)}]
\item For $L^{10}_{2;1}$,
\(
\alpha(e_1)=0,
\
\alpha(e_2)=e_2,
\
\alpha(e_3)=0,
\)
and, for every $k\geq1$,
\[
\mathrm{\operatorname{Der}}_{\alpha^k}(\mathfrak{g})
=
\left\{
\begin{pmatrix}
m&0&0\\
0&r&0\\
0&0&0
\end{pmatrix}
\;\middle|\;
m,r\in\mathbb{K}
\right\}.
\]
This algebra is not complete.

\item For $L^{11,a}_{2;1}$,
\(
\alpha(e_1)=0,
\
\alpha(e_2)=ae_2,
\
\alpha(e_3)=0,
\
a\neq1.
\)
If $a=0$, then
\[
\mathrm{\operatorname{Der}}_{\alpha^k}(\mathfrak{g})_{\bar0}
=
\left\{
\begin{pmatrix}
0&n&0\\
0&r&0\\
0&0&0
\end{pmatrix}
\;\middle|\;
n,r\in\mathbb{K}
\right\},
\]
and
\[
\mathrm{\operatorname{Der}}_{\alpha^k}(\mathfrak{g})_{\bar1}
=
\left\{
\begin{pmatrix}
0&0&0\\
0&0&0\\
0&u&0
\end{pmatrix}
\;\middle|\;
u\in\mathbb{K}
\right\}.
\]
If $a\neq0,1$, then
\[
\mathrm{\operatorname{Der}}_{\alpha^k}(\mathfrak{g})
=
\left\{
\begin{pmatrix}
m&0&0\\
0&r&0\\
0&0&0
\end{pmatrix}
\;\middle|\;
m,r\in\mathbb{K},
\quad
(1-a^k)m=0
\right\}.
\]
No member of $L^{11,a}_{2;1}$ is complete.

\item For $L^{12,a}_{2;1}$,
\[
\alpha(e_1)=ae_1,
\qquad
\alpha(e_2)=e_2,
\qquad
\alpha(e_3)=0,
\qquad
a\neq1.
\]
If $a=0$, then
\[
\mathrm{\operatorname{Der}}_{\alpha^k}(\mathfrak{g})
=
\left\{
\begin{pmatrix}
m&0&0\\
0&r&0\\
0&0&0
\end{pmatrix}
\;\middle|\;
m,r\in\mathbb{K}
\right\},
\]
and the algebra is not complete.

If $a\neq0,1$, then
\[
\mathrm{\operatorname{Der}}_{\alpha^k}(\mathfrak{g})
=
\left\{
\begin{pmatrix}
m&0&0\\
0&0&0\\
0&0&v
\end{pmatrix}
\;\middle|\;
m,v\in\mathbb{K},
\quad
(1-a^k)v=0
\right\},
\]
and
\[
\mathrm{Inn}_{\alpha^{k-1}}(\mathfrak{g})
=
\left\{
\begin{pmatrix}
m&0&0\\
0&0&0\\
0&0&0
\end{pmatrix}
\;\middle|\;
m\in\mathbb{K}
\right\}.
\]
Consequently,
\[
L^{12,a}_{2;1}\text{ is complete}
\quad\Longleftrightarrow\quad
a^k\neq1\ \text{for every integer }k\geq1.
\]
\end{enumerate}
\end{lemma}

\begin{proof}
The assertions follow by substituting $D_{\bar0}$ and $D_{\bar1}$
into the commutation relation and derivation identity.

For $L^{10}_{2;1}$, the derivation equations give
\[
n=q=v=p=s=t=u=0,
\]
whereas $m$ and $r$ are arbitrary. The derivation with
$D(e_2)=e_2$ is outer.

For $L^{11,a}_{2;1}$, the case $a=0$ has the additional maps displayed
in the statement. When $a\neq0,1$, the equations reduce to
\[
n=q=v=p=s=t=u=0,
\qquad
(1-a^k)m=0,
\]
whereas $r$ is arbitrary. Since $\alpha$ has no nonzero fixed point,
all inner derivations vanish, whereas the derivation with
$D(e_2)=e_2$ is nonzero.

For $L^{12,a}_{2;1}$ with $a\neq0,1$, the only fixed homogeneous elements are the multiples of $e_2$. Moreover,
\(
\operatorname{ad}_{ce_2}^{\,k-1}(e_1)
=
[ce_2,\alpha^{k-1}(e_1)]
=
-ca^{k-1}e_1,
\)
and this inner derivation vanishes on $e_2$ and $e_3$. Hence, the inner space is exactly the one displayed above. The equality
\(
\mathrm{\operatorname{Der}}_{\alpha^k}(\mathfrak{g})
=
\mathrm{Inn}_{\alpha^{k-1}}(\mathfrak{g})
\)
holds for every $k\geq1$ precisely when $a^k\neq1$ for every
$k\geq1$.
\end{proof}

\begin{lemma}\label{lem:L16-L20}
Assume
\[
[e_1,e_2]=\gamma e_1,
\qquad
[e_2,e_3]=\lambda e_3,
\qquad
[e_1,e_3]=[e_3,e_3]=0,
\qquad
\gamma\lambda\neq0.
\]
Then, none of the families $L^{16}_{2;1}$--$L^{20}_{2;1}$ is complete.
Their derivation spaces are as follows.

\begin{enumerate}[label=\textup{\rm(\roman*)}]
\item For $L^{16,a}_{2;1}$,
\(
\alpha(e_1)=0,
\
\alpha(e_2)=e_2,
\
\alpha(e_3)=ae_3,
\
a\neq0,
\)
and
\[
\mathrm{\operatorname{Der}}_{\alpha^k}(\mathfrak{g})_{\bar0}
=
\left\{
\begin{pmatrix}
m&0&0\\
0&0&0\\
0&0&v
\end{pmatrix}
\;\middle|\;
m,v\in\mathbb{K}
\right\},
\]
\[
\mathrm{\operatorname{Der}}_{\alpha^k}(\mathfrak{g})_{\bar1}
=
\begin{cases}
\displaystyle
\left\{
\begin{pmatrix}
0&0&0\\
0&0&0\\
0&u&0
\end{pmatrix}
\;\middle|\;
u\in\mathbb{K}
\right\},
&a=1,\\[2.2em]
\{0\},
&a\neq1.
\end{cases}
\]

\item For $L^{17}_{2;1}$,
\(
\alpha(e_1)=0,
\
\alpha(e_2)=e_2,
\
\alpha(e_3)=0,
\)
and
\[
\mathrm{\operatorname{Der}}_{\alpha^k}(\mathfrak{g})_{\bar0}
=
\left\{
\begin{pmatrix}
m&0&0\\
0&r&0\\
0&0&v
\end{pmatrix}
\;\middle|\;
m,r,v\in\mathbb{K}
\right\},
\]
whereas
\[
\mathrm{\operatorname{Der}}_{\alpha^k}(\mathfrak{g})_{\bar1}
=
\begin{cases}
\displaystyle
\left\{
\begin{pmatrix}
0&0&p\\
0&0&0\\
t&0&0
\end{pmatrix}
\;\middle|\;
p,t\in\mathbb{K}
\right\},
&\gamma+\lambda=0,\\[2.2em]
\{0\},
&\gamma+\lambda\neq0.
\end{cases}
\]

\item For $L^{18,a,b}_{2;1}$,
\(
\alpha(e_1)=0,
\
\alpha(e_2)=ae_1+be_2,
\
\alpha(e_3)=0,
\
b\neq1.
\)
For every $k\geq1$,
\[
\mathrm{\operatorname{Der}}_{\alpha^k}(\mathfrak{g})_{\bar0}
=
\left\{
\begin{pmatrix}
m&n&0\\
0&r&0\\
0&0&v
\end{pmatrix}
\;\middle|\;
\begin{array}{l}
a(m-r)+bn=0,\\
(1-b^k)m=0,\\
(1-b^k)v=0
\end{array}
\right\},
\]
and
\[
\mathrm{\operatorname{Der}}_{\alpha^k}(\mathfrak{g})_{\bar1}
=
\left\{
\begin{pmatrix}
0&0&p\\
0&0&0\\
t&u&0
\end{pmatrix}
\;\middle|\;
\begin{array}{l}
at+bu=0,\\
(\gamma+\lambda b^k)t=0,\\
(\lambda+\gamma b^k)p=0
\end{array}
\right\}.
\]

\item For $L^{19,a,b}_{2;1}$,
\(
\alpha(e_1)=ae_1,
\
\alpha(e_2)=be_1+e_2,
\
\alpha(e_3)=0,
\
a\neq0,
\)
and
\[
\mathrm{\operatorname{Der}}_{\alpha^k}(\mathfrak{g})
=
\left\{
\begin{pmatrix}
m&n&0\\
0&0&0\\
0&0&v
\end{pmatrix}
\;\middle|\;
m,n,v\in\mathbb{K},
\quad
(1-a)n+bm=0
\right\}.
\]

\item For $L^{20,a,b,c}_{2;1}$,
\(
\alpha(e_1)=ae_1,
\
\alpha(e_2)=be_1+e_2,
\
\alpha(e_3)=ce_3,
\
ac\neq0,
\)
and
\[
\mathrm{\operatorname{Der}}_{\alpha^k}(\mathfrak{g})_{\bar0}
=
\left\{
\begin{pmatrix}
m&n&0\\
0&0&0\\
0&0&v
\end{pmatrix}
\;\middle|\;
m,n,v\in\mathbb{K},
\quad
(1-a)n+bm=0
\right\},
\]
\[
\mathrm{\operatorname{Der}}_{\alpha^k}(\mathfrak{g})_{\bar1}
=
\left\{
\begin{pmatrix}
0&0&p\\
0&0&0\\
t&u&0
\end{pmatrix}
\;\middle|\;
\begin{array}{l}
(c-a)p=0,\quad(\gamma+\lambda)p=0,\\
(a-c)t=0,\quad(\gamma+\lambda)t=0,\\
bt+(1-c)u=0
\end{array}
\right\}.
\]
\end{enumerate}
\end{lemma}

\begin{proof}
Each formula is obtained by solving the linear equations resulting from
\[
D\circ\alpha=\alpha\circ D
\]
and the $\alpha^k$-derivation identity on the ordered pairs
\[
(e_1,e_2),\quad(e_2,e_1),\quad(e_2,e_3),
\quad(e_3,e_2),
\]
together with the zero brackets.

For $L^{16,a}_{2;1}$, the even derivation space is two-dimensional,
whereas its even inner derivation space has dimension at most one.
For $L^{17}_{2;1}$, the even derivation space is three-dimensional,
whereas the inner derivation space has dimension at most one.

For $L^{18,a,b}_{2;1}$, the assumption $b\neq1$ implies that
$\alpha(x)=x$ only for $x=0$. Hence, \(\mathrm{Inn}_{\alpha^{k-1}}(\mathfrak{g})=\{0\},\) whereas the displayed even derivation space is nonzero.

For $L^{19,a,b}_{2;1}$ and $L^{20,a,b,c}_{2;1}$, the displayed even
derivation spaces have dimension at least two, while their even inner
derivation spaces have strictly smaller dimensions. Therefore, none of these Hom-Lie superalgebras is complete.
\end{proof}

\begin{lemma}\label{lem:L23}
Let $L^{23}_{2;1}$ be defined by
\[
[e_1,e_2]=\gamma e_1,
\qquad
[e_2,e_3]=\lambda e_3,
\qquad
[e_3,e_3]=\mu e_1,
\qquad
\gamma\lambda\mu\neq0,
\]
and
\(
\alpha(e_1)=0,
\
\alpha(e_2)=e_1,
\
\alpha(e_3)=0.
\)
Then, for every integer $k\geq1$,
\[
\mathrm{\operatorname{Der}}_{\alpha^k}(\mathfrak{g})_{\bar0}
=
\left\{
\begin{pmatrix}
0&n&0\\
0&0&0\\
0&0&0
\end{pmatrix}
\;\middle|\;
n\in\mathbb{K}
\right\},
\]
and
\[
\mathrm{\operatorname{Der}}_{\alpha^k}(\mathfrak{g})_{\bar1}
=
\left\{
\begin{pmatrix}
0&0&0\\
0&0&0\\
0&u&0
\end{pmatrix}
\;\middle|\;
u\in\mathbb{K}
\right\}.
\]
Moreover,
\(
\mathrm{Inn}_{\alpha^{k-1}}(\mathfrak{g})=\{0\},
\)
so $L^{23}_{2;1}$ is not complete.
\end{lemma}

\begin{proof}
The commutation relation and the derivation identity give
\begin{align*}
& D_{\bar0}(e_1)=D_{\bar0}(e_3)=0,
\quad
D_{\bar0}(e_2)=ne_1,
\\
& D_{\bar1}(e_1)=D_{\bar1}(e_3)=0,
\quad
D_{\bar1}(e_2)=ue_3.
\end{align*}
The equation $\alpha(x)=x$ has only the zero solution. Thus, every inner derivation is zero, whereas the displayed derivation spaces are nonzero.
\end{proof}

\begin{theorem}\label{complete3dim}
Among the $3$-dimensional complex multiplicative Hom-Lie
superalgebras of superdimension $(2|1)$ listed in
\textnormal{\cite{wang2016}}, the complete ones are precisely
\[
L^{12,a}_{2;1}:
\qquad
[e_1,e_2]=e_1,
\qquad
[e_1,e_3]=e_3,
\qquad
\alpha=
\begin{pmatrix}
a&0&0\\
0&1&0\\
0&0&0
\end{pmatrix},
\]
where
\(
a\in\mathbb{C}^{*}
\)
and
\(
a^k\neq1
\)
for each integer \(k\geq 1\).

Equivalently, $a$ is nonzero and not a root of unity.
\end{theorem}

\begin{proof}
By Lemma~\ref{lem:center-screening-3dim}, all families with nonzero
center are excluded. Lemmas~\ref{lem:L6-L9},
\ref{lem:L10-L12}, \ref{lem:L16-L20}, and \ref{lem:L23}
show that every remaining centerless family is not complete, except
for $L^{12,a}_{2;1}$ with
\(
a^k\neq1
\)
for all \(k\geq1.\)
For this family, Lemma~\ref{lem:L10-L12} gives
\(
\mathrm{\operatorname{Der}}_{\alpha^k}(\mathfrak{g})
=
\mathrm{Inn}_{\alpha^{k-1}}(\mathfrak{g})
\)
for every $k\geq1$, and its center is zero. Hence it is complete.
\end{proof}

\begin{remark}
The parameter range for $L^{12,a}_{2;1}$ given in
\cite{wang2016} contains a redundancy at $a=0$, because that value
coincides with $L^{10}_{2;1}$. This does not affect Theorem~\ref{complete3dim} because
$L^{12,0}_{2;1}=L^{10}_{2;1}$ is not complete.
\end{remark}

\begin{definition}\label{def:Mor-nary}
Let $\mathfrak{g}$ be a $\mathbb{Z}_2$-graded linear space and let
\(
[\cdot,\ldots,\cdot]:
\mathfrak{g}^{\times n}\rightarrow\mathfrak{g}
\)
be an even $n$-linear super-skew-symmetric map. We denote by
\(
\mathrm{Mor}
\bigl(\mathfrak{g},[\cdot,\ldots,\cdot]\bigr)
\)
the set of all even linear maps
$\alpha:\mathfrak{g}\rightarrow\mathfrak{g}$ for which
\(
\bigl(
\mathfrak{g},
[\cdot,\ldots,\cdot],
\alpha
\bigr)
\)
is a multiplicative $n$-ary Hom-Lie superalgebra.
\end{definition}

\begin{proposition}\label{prop:Mor-conjugation}
Let $\mathfrak{g}$ be a $\mathbb{Z}_2$-graded linear space, and let
\[
[\cdot,\ldots,\cdot]_1,
\qquad
[\cdot,\ldots,\cdot]_2:
\mathfrak{g}^{\times n}\longrightarrow\mathfrak{g}
\]
be two even $n$-linear super-skew-symmetric maps. Suppose that there
exists an even invertible linear map
$\phi:\mathfrak{g}\rightarrow\mathfrak{g}$ such that
\[
\phi([x_1,\ldots,x_n]_1)
=
[\phi(x_1),\ldots,\phi(x_n)]_2
\]
for all homogeneous $x_1,\ldots,x_n\in\mathfrak{g}$. Then
\[
\mathrm{Mor}
\bigl(\mathfrak{g},[\cdot,\ldots,\cdot]_2\bigr)
=
\left\{
\phi\alpha\phi^{-1}
\;\middle|\;
\alpha\in
\mathrm{Mor}
\bigl(\mathfrak{g},[\cdot,\ldots,\cdot]_1\bigr)
\right\}.
\]
\end{proposition}

\begin{proof}
Let
\(
\alpha\in
\mathrm{Mor}
\bigl(\mathfrak{g},[\cdot,\ldots,\cdot]_1\bigr)
\)
and set $\beta=\phi\alpha\phi^{-1}$. Transporting the
super-Hom-Nambu identity through the even isomorphism $\phi$ shows
that
\[
\bigl(
\mathfrak{g},
[\cdot,\ldots,\cdot]_2,
\beta
\bigr)
\]
satisfies the super-Hom-Nambu identity. Moreover,
\[
\begin{aligned}
\beta([x_1,\ldots,x_n]_2)
&=
\phi\alpha
\bigl(
[\phi^{-1}(x_1),\ldots,\phi^{-1}(x_n)]_1
\bigr)\\
&=
[\beta(x_1),\ldots,\beta(x_n)]_2.
\end{aligned}
\]
Thus,
\(
\phi\alpha\phi^{-1}
\in
\mathrm{Mor}
\bigl(\mathfrak{g},[\cdot,\ldots,\cdot]_2\bigr).
\)
This proves one inclusion. Applying the same argument to
$\phi^{-1}$ proves the reverse inclusion.
\end{proof}

\begin{theorem}\label{2d-3homlie}
Let
\(
\mathfrak{g}_{\bar0}=\mathbb{C}e_0,
\
\mathfrak{g}_{\bar1}=\mathbb{C}e_1.
\)
Every $2$-dimensional complex multiplicative $3$-Hom-Lie
superalgebra of superdimension $(1|1)$ is isomorphic to precisely one
of the following families.

\begin{enumerate}[label=\textup{\arabic*.}]
\item
$\mathfrak{g}^{1,a,b}_{1;1}$ is abelian, with
\(
\alpha(e_0)=ae_0,
\
\alpha(e_1)=be_1,
\
a,b\in\mathbb{C}.
\)

\item
$\mathfrak{g}^{2,a}_{1;1}$ is defined by
\[
[e_1,e_1,e_1]=e_1,
\qquad
\alpha(e_0)=ae_0,
\qquad
\alpha(e_1)=0,
\qquad
a\in\mathbb{C}.
\]

\item
$\mathfrak{g}^{3}_{1;1}$ is defined by
\(
[e_1,e_1,e_0]=e_0,
\
\alpha=0.
\)

\item
$\mathfrak{g}^{4,\tau}_{1;1}$ is defined by
\[
[e_1,e_1,e_0]=e_0,
\qquad
[e_1,e_1,e_1]=\tau e_1,
\qquad
\alpha=0,
\qquad
\tau\in\mathbb{C}^{*}.
\]
Moreover,
\(
\mathfrak{g}^{4,\tau}_{1;1}
\cong
\mathfrak{g}^{4,\tau'}_{1;1}
\Longleftrightarrow
\tau=\tau'.
\)
\end{enumerate}
\end{theorem}

\begin{proof}
By parity and super-skew-symmetry, the only possibly nonzero basic
brackets are
\[
[e_1,e_1,e_0]=ce_0,
\qquad
[e_1,e_1,e_1]=de_1.
\]
Every even linear twisting map has the form
\[
\alpha(e_0)=ae_0,
\qquad
\alpha(e_1)=be_1.
\]
Multiplicativity and the super-Hom-Nambu identity are equivalent to
\[
ac(1-b^2)=0,
\qquad
bd(1-b^2)=0,
\]
\[
bc(ad-3bc)=0,
\qquad
abcd=0,
\qquad
b^2d^2=0.
\]

The first two equations follow from multiplicativity applied to
\[
[e_1,e_1,e_0]=ce_0
\qquad\text{and}\qquad
[e_1,e_1,e_1]=de_1,
\]
respectively. By multilinearity, it is sufficient to verify the
super-Hom-Nambu identity on homogeneous basis elements. The choices
\[
(x_1,x_2;y_1,y_2,y_3)
=
(e_0,e_1;e_1,e_1,e_1),
\]
\[
(e_1,e_1;e_0,e_1,e_1),
\qquad\text{and}\qquad
(e_1,e_1;e_1,e_1,e_1)
\]
give, respectively,
\(
bc(ad-3bc)=0,\
abcd=0,\
b^2d^2=0.
\)
All remaining choices either give zero identically or reproduce these
relations up to a nonzero scalar factor or sign.

If $c=d=0$, then the algebra is abelian. If $c=0$ and $d\neq0$, then
$b=0$, and a change of the odd basis element normalizes $d$ to $1$.
If $c\neq0$ and $d=0$, the displayed equations imply $a=b=0$, and
$c$ can be normalized to $1$. Finally, if $cd\neq0$, then
$a=b=0$; after normalizing $c$ to $1$, the invariant ratio
\(
\tau=\frac{d}{c}
\)
remains. As the ground field is $\mathbb{C}$, the required square roots for these normalizations exist.

An even isomorphism has the form
\[
\phi(e_0)=xe_0,
\qquad
\phi(e_1)=ye_1,
\qquad
xy\neq0.
\]
Both $c$ and $d$ are multiplied by $y^2$ under this change of basis,
so $d/c$ is invariant. This proves the stated isomorphism assertions
and exhaustiveness of the list.
\end{proof}

\begin{theorem}\label{derivations-2d-3homlie}
For every integer $k\geq1$, the spaces of $\alpha^k$-derivations of the algebras in Theorem~\ref{2d-3homlie} are as follows.
\begin{enumerate}[label=\textup{\arabic*.}]
\item For the abelian algebra $\mathfrak{g}^{1,a,b}_{1;1}$,
\[
\mathrm{\operatorname{Der}}_{\alpha^k}(\mathfrak{g})_{\bar0}
=
\left\{
\begin{pmatrix}
m&0\\
0&q
\end{pmatrix}
\;\middle|\;
m,q\in\mathbb{C}
\right\},
\]
and
\[
\mathrm{\operatorname{Der}}_{\alpha^k}(\mathfrak{g})_{\bar1}
=
\begin{cases}
\displaystyle
\left\{
\begin{pmatrix}
0&n\\
p&0
\end{pmatrix}
\;\middle|\;
n,p\in\mathbb{C}
\right\},
&a=b,\\[2.2em]
\{0\},
&a\neq b.
\end{cases}
\]

\item For $\mathfrak{g}^{2,a}_{1;1}$,
\[
\mathrm{\operatorname{Der}}_{\alpha^k}(\mathfrak{g})_{\bar0}
=
\left\{
\begin{pmatrix}
m&0\\
0&0
\end{pmatrix}
\;\middle|\;
m\in\mathbb{C}
\right\},
\]
and
\[
\mathrm{\operatorname{Der}}_{\alpha^k}(\mathfrak{g})_{\bar1}
=
\begin{cases}
\displaystyle
\left\{
\begin{pmatrix}
0&0\\
p&0
\end{pmatrix}
\;\middle|\;
p\in\mathbb{C}
\right\},
&a=0,\\[2.2em]
\{0\},
&a\neq0.
\end{cases}
\]

\item For $\mathfrak{g}^{3}_{1;1}$,
\[
\mathrm{\operatorname{Der}}_{\alpha^k}(\mathfrak{g})_{\bar0}
=
\left\{
\begin{pmatrix}
0&0\\
0&q
\end{pmatrix}
\;\middle|\;
q\in\mathbb{C}
\right\},
\]
and
\[
\mathrm{\operatorname{Der}}_{\alpha^k}(\mathfrak{g})_{\bar1}
=
\left\{
\begin{pmatrix}
0&n\\
0&0
\end{pmatrix}
\;\middle|\;
n\in\mathbb{C}
\right\}.
\]

\item For $\mathfrak{g}^{4,\tau}_{1;1}$,
\(
\mathrm{\operatorname{Der}}_{\alpha^k}(\mathfrak{g})=\{0\}.
\)
\end{enumerate}

Therefore, the complete $2$-dimensional complex multiplicative
$3$-Hom-Lie superalgebras of superdimension $(1|1)$ are precisely
\(
\mathfrak{g}^{4,\tau}_{1;1},
\
\tau\in\mathbb{C}^{*}.
\)
\end{theorem}

\begin{proof}
The abelian case follows from the commutation relation
$D\circ\alpha=\alpha\circ D$.
For each nonabelian algebra, substitute a homogeneous matrix
\[
D_{\bar0}
=
\begin{pmatrix}
m&0\\
0&q
\end{pmatrix},
\qquad
D_{\bar1}
=
\begin{pmatrix}
0&n\\
p&0
\end{pmatrix}
\]
into the $\alpha^k$-derivation identity. This gives the displayed
spaces.
The first two families have a nonzero center. The algebra
$\mathfrak{g}^{3}_{1;1}$ has zero center, but its derivation space is
nonzero. Since $\alpha=0$, it has no nonzero fixed fundamental object,
and hence
\(
\mathrm{Inn}_{\alpha^{k-1}}(\mathfrak{g})=\{0\}.
\)
Finally, for $\mathfrak{g}^{4,\tau}_{1;1}$, the bracket is surjective.
Thus, every $\alpha^k$-derivation vanishes for $k\geq1$. Its center is
zero, and its inner derivation space is zero. Therefore, $\mathfrak{g}^{4,\tau}_{1;1}$ is complete.
\end{proof}

\begin{remark}
Theorem~\ref{2d-3homlie} classifies the superdimension $(1|1)$ case.
No classification of the purely odd superdimension $(0|2)$ case is
claimed here.
\end{remark}

\section{Conclusions}\label{sec5-Conclusions}

In this work, we introduced and studied completeness for multiplicative
\(n\)-ary Hom-Nambu superalgebras. Since a Hom-Nambu bracket need not
be totally super-skew-symmetric, the center was defined as the intersection
of the positional centers. This definition reduces to the usual center in
the \(n\)-Hom-Lie case.

We proved a completeness criterion for multiplicative \(n\)-ary
Hom-Nambu superalgebras with zero twisting map and surjective bracket,
and we exhibited a complete \(n\)-ary Hom-Nambu superalgebra that is
not an \(n\)-Hom-Lie superalgebra. We also established direct-sum results
for complete \(n\)-Hom-Lie superalgebras with bijective twisting maps.

Starting from a multiplicative Hom-Lie superalgebra, we considered the
recursively induced multiplicative \(n\)-ary Hom-Nambu superalgebras
\[
\mathfrak{g}_n
=
\bigl(
\mathfrak{g},
[\cdot,\ldots,\cdot]_n,
\beta_n
\bigr),
\qquad
\beta_n=\alpha^{n-1}.
\]
When \(\alpha\) is surjective, we proved that the center of the original
binary Hom-Lie superalgebra is trivial if and only if the center
\(C(\mathfrak{g}_n)=\bigcap_{i=1}^{n}C_i(\mathfrak{g}_n)\)
of each induced \(n\)-ary Hom-Nambu superalgebra is trivial. We also
recalled that a binary \(\alpha^k\)-derivation satisfies the corresponding
relative derivation identity for the recursively induced bracket. This
identity alone does not imply a complete correspondence between the
derivation spaces of the binary and induced \(n\)-ary structures.

Finally, we determined the complete members in selected low-dimen\-sional
classifications of multiplicative Hom-Lie and \(3\)-Hom-Lie
superalgebras. Further work may investigate additional hypotheses under which
completeness, or the equality of the corresponding derivation and
inner-derivation spaces, is preserved under twisting or recursive
higher-arity constructions.




\section*{Acknowledgements}
\medskip
 \noindent\textbf{Funding}
M. R. Farhangdoost and M. R. Hafezi would like to thank Shiraz University for financial support.
This research was supported by Grant No. 98GRC1M82582, Shiraz University, Shiraz, Iran.


\medskip
\noindent\textbf{Data Availability Statement}
Data sharing is not applicable to this article, as no datasets were generated or analyzed during the current study.

\medskip
\noindent\textbf{Declarations}
\textbf{Conflict of interest} The authors have no conflicts of interest to disclose.



\begin{thebibliography}{99}
\bibitem{AbramovSilvestrov:3homLiealgsigmaderivINvol}
Abramov, V., Silvestrov, S.: $3$-Hom-Lie algebras based on
$\sigma$-derivation and involution. Adv. Appl. Clifford Algebras, 30, 45 (2020)  

\bibitem{AbdaouiMabrMakhl2014ArXiv-CohomlHomLeibnAryHomNambuLieSuperAlg} Abdaoui, K., Mabrouk, S., Makhlouf, A.: Cohomology of Hom-Leibniz and $n$-ary Hom-Nambu-Lie superalgebras. arXiv:1406.3776 [math.RT].\\

\bibitem{ammarmakhlouf2010super}
	Ammar, F., Makhlouf, A.:
	Hom-Lie superalgebras and Hom-Lie admissible superalgebras.
	J. Algebra, 324(7), 1513--1528 (2010)

\bibitem{AmmarMabMakh2011-naryhomrep}
	Ammar, F., Mabrouk, S., Makhlouf, A.:
	Representation and cohomology of $n$-ary multiplicative Hom-Nambu-Lie algebras.
	J. Geom. Phys. 61, 1898--1913 (2011)


\bibitem{ArnlindMkhSilvJMP2010:ternaryHomNambuLieINdHomLie}
Arnlind, J., Makhlouf, A., Silvestrov, S.:
Ternary Hom-Nambu-Lie algebras induced by Hom-Lie algebras. J. Math. Phys. 51, 043515
(2010)
	
\bibitem{arnlindMakhlSilvJMP2011-ConstrnLienHomNambuLie} Arnlind, J., Makhlouf, A., Silvestrov, S.:
Construction of $n$-Lie algebras and $n$-ary Hom-Nambu-Lie algebras. J. Math. Phys. 52, 123502 (2011)
	
\bibitem{akms:ternary}
Arnlind, J., Kitouni, A., Makhlouf, A., Silvestrov, S.: Structure and Cohomology of $3$-Lie Algebras Induced by Lie Algebras. In: Makhlouf, A., Paal, E., Silvestrov, S., Stolin, A. (eds) Algebra, Geometry and Mathematical Physics. Springer Proceedings in Mathematics and Statistics, vol 85, pp. 123--144, Springer, Berlin, Heidelberg (2014) 

\bibitem{armakan2020complete}
	Armakan, A., Razavi, A.:
	Complete Hom-Lie superalgebras.
	Commun. Algebra, 48(2), 651--662 (2020)

\bibitem{ArmakanSilvFarhangTJM2019EnvAlgColHomLie}
	Armakan, A., Silvestrov, S., Farhangdoost, M. R.:
	Enveloping algebras of color Hom-Lie algebras.
	Turk. J. Math. 43(1), 316--339  (2019)
	
\bibitem{ArmakanSilvFarhangGMJ2021ExtHomLieColAlg}
	Armakan, A., Silvestrov, S., Farhangdoost, M. R.:
	Extensions of Hom-Lie color algebras.
	Georgian Math. J. 28(1), 15--27 (2021)

\bibitem{AtMaSi:GenNambuAlg}
Ataguema, H., Makhlouf, A., Silvestrov, S.: Generalization of $n$-ary Nambu
algebras and beyond. J. Math. Phys. 50, 083501 (2009)

\bibitem{awataLiMincYoneyaJHEP2001-quantnambubr}
Awata, H., Li, M., Minic, D., Yoneya, T.: On the quantization of Nambu brackets.
J. High Energy Phys. 2, 013 (2001). DOI:10.1088/1126-6708/2001/02/013

\bibitem{bai2010}
Bai, R., Bai, C., Wang, J.:
Realizations of 3-Lie algebras.
J. Math. Phys. 51, 063505 (2010)
	
\bibitem{bai2011}
	Bai, R., Song, G., Zhang, Y.:
	On classification of $n$-Lie algebras.
	Front. Math. China, 6, 581--606 (2011)

\bibitem{bakayoko2014a}
    Bakayoko, I.:
	Laplacian of Hom-Lie quasi-bialgebras.
	Int. J. Algebra, 8(15), 713--727 (2014)
	
\bibitem{bakayoko2014b}
	Bakayoko, I.:
	L-modules, L-comodules and Hom-Lie quasi-bialgebras.
	African Diaspora J. Math., 17, 49--64 (2014)

\bibitem{BakayokoSilvestrov-SPROMS2020-MultnHomLiecolalg}
Bakayoko, I., Silvestrov, S.: Multiplicative $n$-Hom-Lie color algebras.
In: Silvestrov, S., Malyarenko, A., Ran\v{c}i\'{c}, M. (eds) Algebraic Structures and Applications. SPAS 2017. Springer Proceedings in Mathematics and Statistics, vol 317, pp. 159--187, Springer, Cham, Switzerland (2020).
(arXiv:1912.10216 [math.QA])
	
\bibitem{BakayokoSilv-AfM2021-HomleftSymcolHomtridendColAlg}
	Bakayoko, I., Silvestrov, S.:
	Hom-left-symmetric color dialgebras and Hom-tridendriform color algebras.
	Afr. Mat. 32, 941--958 (2021)

\bibitem{BeitesKaygorodovPopovBMMSS2019-GenDernaryHomOmegacolalg}
Beites, P.D., Kaygorodov, I., Popov, Y.:
Generalized derivations of multiplicative $n$-ary Hom-$\Omega$ color algebras.
Bull. Malays. Math. Sci. Soc. 42, 315--335 (2019)

\bibitem{BenAbdeljElhamdKaygorMakhl201920GenDernBiHomLiealg}
Ben Abdeljelil, A., Elhamdadi, M., Kaygorodov, I., Makhlouf, A.
Generalized derivations of $n$-BiHom-Lie algebras. In:  Silvestrov, S., Malyarenko, A., Ran\v{c}i\'{c}, M. (eds) Algebraic Structures and Applications. Springer Proceedings in Mathematics and Statistics, vol 317, Springer, Cham, 81--97 (2020)
(arXiv:1901.09750 [math.RA])

\bibitem{AbdelkhaderSamiOthmen}
Ben Hassine, A., Mabrouk, S., Ncib, O.: Some constructions of multiplicative $n$-ary hom-Nambu algebras. \emph{Adv. Appl. Clifford Algebr.} \textbf{2019}, \emph{29}, 88.

\bibitem{chen2013}
	Chen, L., Ma, Y., Ni, L.:
	Generalized derivations of Lie color algebras.
	Results Math. 63(3--4), 923--936. (2013)

\bibitem{daletskii1997}
	Daletskii, Y., Takhtajan, L.:
	Leibniz and Lie algebra structures for Nambu algebras.
	Letters in Mathematical Physics, 39, 127--141 (1997)

\bibitem{deazcarraga2010}
	De Azcarraga, J. A., Izquierdo, J. M.:
	$n$-Ary algebras: a review with applications.
	J. Phys. A: Math. Theor., 43, 293001 (2010)
    DOI 10.1088/1751-8113/43/29/293001

\bibitem{guan2017}
	Guan, B., Chen, L., Sun, B.:
	3-Ary Hom-Lie superalgebras induced by Hom-Lie superalgebras.
	Adv. Appl. Clifford Algebras, 27, 3063--3082 (2017)
	
\bibitem{guan2019}
	Guan, B., Chen, L., Sun, B.:
	On Hom-Lie superalgebras.
	Adv. Appl. Clifford Algebras, 29(1), 16 (2019)

\bibitem{hartwiglarssonSilvestrov20032006}
Hartwig, J.T., Larsson, D., Silvestrov, S.D.:
Deformations of Lie algebras using $\sigma$-derivations. J. Algebra, 295(2), 314--361 (2006).  (First published as: Preprints in Mathematical Sciences 2003:32, LUTFMA-5036-2003, Centre for Mathematical Sciences, Lund University, 52 pp, 2003. arXiv:math/0408064 [math.QA])

\bibitem{fan2021}
	Fan, Y., Li, J., Chen, L.:
	Complete Bihom-Lie superalgebras and derivation superalgebras.
	Commun. Algebra, 49(5), 1925--1937 (2021)

\bibitem{farhangdoost2023}
	Farhangdoost, M.R., Attari Polsangi, A.R., Silvestrov, S.:
	Simply complete Hom-Lie superalgebras and decomposition of complete Hom-Lie superalgebras. Adv. Appl. Clifford Algebras, 33, 17 (2023)

\bibitem{filippov1985}
Filippov, V. T.: $n$-Lie algebras. Siberian Mathematical Journal, 26, 879--891 (1985)

\bibitem{kaygorodov2012}
	Kaygorodov, I.:
	On $\delta$-Derivations of $n$-ary algebras.
	Izv. Math., 76(5), 1150--1162 (2012)
	
\bibitem{kaygorodov2016}
	Kaygorodov, I., Popov, Y.:
	Generalized derivations of (color) $n$-ary algebras.
    Linear and Multilinear Algebra, 64(6), 817--837 (2016)

\bibitem{km:nary}
Kitouni, A., Makhlouf, A.: On structure and central extensions of $(n+1)$-Lie algebras induced by $n$-Lie algebras. arXiv:1405.5930 [math.RA] (2014)

\bibitem{KitMakSil-GMJ2016nhominduced}
Kitouni, A., Makhlouf, A., Silvestrov, S.:
	On $(n+1)$-Hom-Lie algebras induced by $n$-Hom-Lie algebras.
	Georgian Math. J., 23(1), 75--95 (2016). (arXiv:1504.06980 [math.RA])

\bibitem{kms:narygenBiHomLieBiHomassalgebras2020}
Kitouni, A.; Makhlouf, A.; Silvestrov, S. On $n$-ary generalization of BiHom-Lie algebras and BiHom-associative algebras. In: Silvestrov, S., Malyarenko, A., Ran\v{c}i\'{c}, M., (eds) Algebraic Structures and Applications. Springer Proceedings in Mathematics and Statistics, vol 317, Springer, Cham, 99--126 (2020)

\bibitem{kms:solvnilpnhomlie2020}
Kitouni, A.; Makhlouf, A.; Silvestrov, S. On solvability and nilpotency for $n$-Hom-Lie algebras and $(n+1)$-Hom-Lie algebras induced by $n$-Hom-Lie algebras. In: Silvestrov, S., Malyarenko, A., Ran\v{c}i\'{c}, M., (eds) Algebraic Structures and Applications.
Springer Proceedings in Mathematics and Statistics, vol 317, Springer, Cham, 127--157 (2020)

\bibitem{kitmbongsil-2023NAART:IdlsDerCenDescSerNaryHomAlg} Kitouni, A., Mboya, S., Ongong’a, E., Silvestrov, S.: On Ideals and Derived and Central Descending Series of $n$-ary Hom-Algebras. In: Albuquerque, H., Brox, J., Mart{\'i}nez, C., Saraiva, P. (eds) Non-Associative Algebras and Related Topics. NAART 2020. Springer Proceedings in Mathematics and Statistics, vol 427, Springer, Cham. 261–286, (2023)

\bibitem{cln1dimnHomLieniltw}
Kitouni, A.; Silvestrov, S. On Classification of $(n+1)$-dimensional $n$-Hom-Lie algebras with nilpotent twisting maps. In: Silvestrov, S., Malyarenko, A. (eds) Non-commutative and Non-associative Algebra and Analysis Structures. SPAS 2019. Springer Proceedings in Mathematics and Statistics, vol 426, Springer, Cham. 525--562 (2023) 

\bibitem{LarssonSilv2005:QuasiLieAlg}
Larsson, D., Silvestrov, S. D.: Quasi-Lie algebras. In: Fuchs, J., Mickelsson, J., Rozenblioum, G., Stolin, A., Westerberg, A. (eds),
Noncommutative Geometry and Representation Theory in Mathematical Physics. Contemp. Math. \textbf{391}, Amer. Math. Soc., Providence, RI, 241-248 (2005) (First published as: Preprints in Mathematical Sciences 2004:30, LUTFMA-5049-2004, Centre for Mathematical Sciences, Lund University (2004))

\bibitem{LarssonSilv:GradedquasiLiealg}
Larsson, D., Silvestrov, S. D.: Graded quasi-Lie algebras. Czech. J. Phys. 55(11),
1473--1478 (2005) 

\bibitem{LarssonSigSilvJGLTA2008:QuasiLiedefFttN}
Larsson, D., Sigurdsson, G., Silvestrov, S. D.: Quasi-Lie deformations on the algebra $\mathbb{F}[t]/(t^N)$. J. Gen. Lie Theory Appl. \textbf{2}(3), 201-205 (2008)

\bibitem{LarssonSilvJA2005:QuasiHomLieCentExt2cocyid}
Larsson, D., Silvestrov, S. D.: Quasi-Hom-Lie algebras, central extensions and $2$-cocycle-like identities. J. Algebra, \textbf{288}, 321-344 (2005) (First published as: Preprints in Mathematical Sciences 2004:3, LUTFMA-5038-2004, Centre for Mathematical Sciences, Department of Mathematics, Lund Institute of Technology, Lund University (2004))

\bibitem{LarssonSilv20042005CA2007:quasidefsl2}
Larsson, D., Silvestrov, S. D.: Quasi-deformations of $sl_2(\mathbb{F})$ using twisted derivations. Comm. Algebra, \textbf{35}(12), 4303--4318 (2007)
(First published as:
Preprints in Mathematical Sciences 2004:26, LUTFMA-5047-2004, Centre for Mathematical Sciences, Lund Institute of Technology, Lund University (2004).
arXiv:math/0506172 [math.RA] (2005))

\bibitem{leger2000}
Leger, G., Luks, E.: Generalized derivations of Lie algebras.
J. Algebra, 228(1), 165--203 (2000)

\bibitem{ma2017a}
Ma, T., Makhlouf, A., Silvestrov, S.: Curved $\mathcal O$-operator systems. Commun. Algebra, 52(12), 5186–5202 (2024)
(First published as: arXiv:1710.05232 [math.RA] (2017))
	
\bibitem{ma2017b}
Ma, T., Makhlouf, A., Silvestrov, S.:
Rota-Baxter bisystems and covariant bialgebras. arXiv:1710.05161 [math.RA]

\bibitem{MaMkhlSilv2021} Ma, T., Makhlouf, A., Silvestrov, S.:
Rota–Baxter cosystems and coquasitriangular mixed bialgebras. J. Algebra Appl. \textbf{20}(04), 2150064 (2021)

\bibitem{MabroukNcibSilvAACA2021-GenDerRotaBaxterOpsnaryHomNambuSuperal}
Mabrouk, S., Ncib, O., Silvestrov, S.: Generalized derivations and Rota-Baxter operators of
$n$-ary Hom-Nambu superalgebras. Adv. Appl. Clifford Algebras \textbf{31}, 32 (2021) (arXiv:2003.01080[math.QA])\\

\bibitem{MakhlSilv2006GLTA2008-Homalgstr}
Makhlouf, A., Silvestrov, S. D.:
Hom-algebra structures. J. Gen. Lie Theory Appl. \textbf{2}(2), 51--64 (2008).
(First published as:
Preprints in Mathematical Sciences  2006:10, LUTFMA-5074-2006, Centre for Mathematical Sciences,
Lund University (2006))
	
\bibitem{MakhSilv2007FM2010-NotesformdefHomassHomLie}
Makhlouf, A., Silvestrov, S.: Notes on $1$-parameter formal deformations of Hom-associative and Hom-Lie algebras, Forum Math. \textbf{22}(4), 715-739 (2010)
(First published as: Preprints in Mathematical Sciences, 2007:31, LUTFMA-5095-2007, Centre for Mathematical Sciences,
Lund University, (2007). arXiv:0712.3130v1 [math.RA])

\bibitem{MakhSil200709HomHopfGLTBSpringer} Makhlouf, A., Silvestrov, S.:
Hom-Lie admissible Hom-coalgebras and Hom-Hopf algebras,
In: Silvestrov, S., Paal, E., Abramov, V., Stolin, A. (Eds.),
Generalized Lie Theory in Mathematics, Physics and Beyond, Springer-Verlag, Berlin, Heidelberg, Ch. 17, 189-206 (2009) (First published as: Preprints in Mathematical Sciences, Lund University, Centre for Mathematical Sciences, Centrum Scientiarum Mathematicarum (2007:25) LUTFMA-5091-2007. arXiv:0709.2413 [math.RA] (2007))

\bibitem{MakhSilvJAA2010HomAlgHomCoalg} Makhlouf, A., Silvestrov, S. D.:
Hom-algebras and Hom-coalgebras, J. Algebra Appl. \textbf{9}(4), 553-589 (2010).
(First published as: Preprints in Mathematical Sciences, Lund University, Centre for Mathematical Sciences, Centrum Scientiarum Mathematicarum, (2008:19) LUTFMA-5103-2008. arXiv:0811.0400
[math.RA] (2008)) 


\bibitem{makhlouf2009}
Makhlouf, A.:
Paradigm of Nonassociative Hom-algebras and Hom-superalgebras. In: Carmona Tapia, J., Morales Campoy, A., Peralta Pereira, A. M., Ram{\'i}rez {\'A}lvarez M. I. (eds), Proceedings of Jordan Structures in Algebra and Analysis Meeting, Publishing house: Circulo Rojo, 145--177 (2010)

\bibitem{nambu1973}
Nambu, Y.:
Generalized Hamiltonian dynamics. Phys. Rev. D, \textbf{7}(8), 2405--2412 (1973)

\bibitem{RichardSilvJA2008:quasiLiesigderCtpm1}
Richard, L., Silvestrov, S. D.: Quasi-Lie structure of $\sigma$-derivations of $\mathbb{C}[t^{\pm1}]$, J. Algebra, \textbf{319}(3), 1285-1304 (2008)

\bibitem{RichardSilvestrovGLTMPBSpr2009:QuasiLieHomLiesigmaderiv}
Richard, L., Silvestrov, S.:
A Note on Quasi-Lie and Hom-Lie structures of $\sigma$-derivations of $\mathbb{C}[z_1^{\pm1},\dots,z_n^{\pm1}]$. In: Silvestrov, S., Paal, E., Abramov, V., Stolin, A. (eds), Generalized Lie Theory in Mathematics, Physics and Beyond. Springer-Verlag, Berlin, Heidelberg, Ch. 22, 257-262 (2009)

\bibitem{SerganovaVaintrob2026}
V.~Serganova and A.~Vaintrob,
\emph{Almost inner derivations of Lie superalgebras},
J. Algebra \textbf{692} (2026), 1--26.

\bibitem{SigSilv:CzechJP2006:GradedquasiLiealgWitt}
Sigurdsson, G., Silvestrov, S.: Graded quasi-Lie algebras of Witt type. Czechoslovak J. Phys. \textbf{56}, 1287-1291 (2006)

\bibitem{SigSilv:GLTbdSpringer2009}
Sigurdsson, G., Silvestrov, S.: Lie color and Hom-Lie algebras of Witt type and their central extensions. In: Silvestrov, S., Paal, E., Abramov, V., Stolin, A. (Eds.), Generalized Lie Theory in Mathematics, Physics and Beyond, Springer-Verlag, Berlin, Heidelberg, Ch. 21, 247-255 (2009)
	
\bibitem{takhtajan1994}
Takhtajan, L. A.:
On the foundation of generalized Nambu mechanics.
Commun. Math. Phys. \textbf{160}(2), 295--315 (1994)
		
\bibitem{takhtajan1995}
	Takhtajan, L. A.:
	Higher order analog of the Chevalley-Eilenberg complex.
	St. Petersburg Math. J. \textbf{6}(2), 429--438 (1995)

\bibitem{VainermanKerner-$n$-ary}
Vainerman, L., Kerner R.: On special classes of $n$-algebras. J. Math. Phys. \textbf{37} (5), 2553-2565 (1996)

\bibitem{wang2016}
	Wang, C., Zhang, Q., Wei, Z.:
	A classification of low-dimensional multiplicative Hom-Lie superalgebras.
	Open Math. \textbf{14}, 613--628 (2016)	

\bibitem{yau2008}
Yau, D.:
Enveloping algebras of Hom-Lie algebras.
J. Gen. Lie Theory Appl. \textbf{2}(2), 95--108 (2008)
	
\bibitem{yau2009}
Yau, D.: Hom-algebras and homology.
\textit{Journal of Lie Theory}, \textbf{19}(2), 409--421 (2009)
	
\bibitem{yau2011}
Yau, D.:
The Hom-Yang-Baxter equation and Hom-Lie algebras.
J. Math. Phys. \textbf{52}(5), 053502 (2011)

\bibitem{YauGenCom}
Yau, D.: A Hom-associative analogue of Hom-Nambu algebras. arXiv:1005.2373 [math.RA] (2010)

\bibitem{YauHomNambuLie}
Yau, D.: On $n$-ary Hom-Nambu and Hom-Nambu-Lie algebras. J. Geom. Phys. \textbf{62}(2), 506--522 (2012)

\bibitem{Yuan2012:HomLiecoloralgstr} Yuan, L.: Hom-Lie color algebra structures, Commun. Algebra, \textbf{40}, 575-592 (2012)

\bibitem{zhang2010}
	Zhang, R., Zhang, Y.:
	Generalized derivations of Lie superalgebras.
	Commun. Algebra, \textbf{38}(10), 3737--3751 (2010)	

\end{thebibliography}
\end{document}